\documentclass[12pt]{amsart}
\usepackage{geometry}                
\usepackage{graphicx}
\usepackage{enumerate}
\usepackage{amssymb}
\usepackage{epstopdf}
\newtheorem{theorem}{Theorem}[section]
\newtheorem{lemma}[theorem]{Lemma} 
\newtheorem{remark}[theorem]{Remark} 
\newtheorem{fact}[theorem]{Fact} 
\newtheorem{definition}[theorem]{Definition} 

\newtheorem{corollary}[theorem]{Corollary} 
 
\newtheorem{example}[theorem]{Example}
\def\da{\downarrow}
\def\nda{\not\downarrow}
\def\A{\mathfrak{A}}
\def\B{\mathfrak{B}}
\def\C{\mathfrak{C}}
\def\D{\mathfrak{D}}
\def\K{\mathcal{K}}
\def\M{\mathbb{M}}
\def\E{\mathcal{E}}
\def\a{\alpha}

\def\o{\omega}
\def\k{\kappa}

\def\raj{\upharpoonright}
\def\v{\vert}

\title{An AEC framework for fields with commuting automorphisms}
\author{Tapani Hyttinen and Kaisa Kangas}

\thanks{Research of the second author was supported by grant 310737 of the Academy of Finland}                                    

\begin{document}
\maketitle

\begin{abstract}
In this paper, we introduce an AEC framework for studying fields with commuting automorphisms.
Fields with commuting automorphisms are closely related to difference fields.
Some authors define a difference ring (or field) as a ring (or field) together with several commuting endomorphisms, while others only study one endomorphism.
Z. Chatzidakis and E. Hrushovski have studied in depth the model theory of ACFA,
the model companion of difference fields with one automorphism.
Our fields with commuting automorphisms generalize this setting.
We have several automorphisms and they are required to commute. 
Hrushovski has proved that in the case of fields with two or more commuting automorphisms,
the existentially closed models do not necessarily form a first order model class.
In the present paper, we introduce FCA-classes, an AEC framework for studying the existentially closed models of the theory of fields with commuting automorphisms.
We prove that an FCA-class has AP and JEP and thus a monster model,
that Galois types coincide with existential types in existentially closed models,
that the class is homogeneous,
and that there is a version of type amalgamation theorem that allows to combine three types under certain conditions.
Finally, we use these results to show that our monster model is a simple homogeneous structure in the sense of S. Buechler and O. Lessman 
(this is a non-elementary analogue for the classification theoretic notion of a simple first order theory). 
\end{abstract}

\tableofcontents

\section{Introduction} 

\noindent
A \emph{field with commuting automorphisms} is a field $(K,+, \cdot)$
together with distinguished field automorphisms $\sigma_1, \ldots, \sigma_n$
that are required to commute, i.e. $\sigma_i \sigma_j=\sigma_j \sigma_i$
for all $i,j=1, \ldots, n$.
Fields with commuting automorphisms arise as a generalisation of fields with one distinguished automorphism.
A \emph{difference field} is sometimes defined to be a field with one distinguished automorphism (see e.g. \cite{cohn2}) but some authors (e.g. \cite{levin}) employ the term to refer to a field with several commuting automorphisms.
A. Macintyre showed in \cite{mac} that the class of existentially closed difference fields with only one automorphism is first-order axiomatizable,
i.e. that the theory of difference fields with one automorphism has a model companion.
Z. Chatzidakis and E. Hrushovski call this model companion ACFA.
They studied its model theory in depth in \cite{ChHr},
and continued this effort together with Y. Peterzil in \cite{CHP}.
Hrushovski used the results on difference fields in his model theoretic proof for the Manin-Mumford conjecture, a statement in arithmetic geometry \cite{manin}.
Moreover, Chatzidakis and Hrushovski have used model theory of difference fields to study algebraic dynamics in \cite{dyn1, dyn2}.
These results highlight the potential for applications in this line of research.
 
In a difference field, there is a geometry of difference varieties where zero sets of difference polynomials generate the closed sets in a Noetherian topology that resembles the Zariski topology on an algebraically closed field (see e.g. \cite{cohn2, levin}).
Here, the existentially closed difference fields play a similar role as algebraically closed fields in algebraic geometry.
In \cite{ChHr}, Chatzidakis and Hrushovski described the dimension theory for ACFA,
including a decomposition into one-dimensional definable sets.
They classified the possible combinatorial geometries underlying the one-dimensional sets and proved that Zilber's Trichotomy holds in characteristic 0.
In \cite{CHP}, they presented a new proof for the trichotomy result that also applies in positive characteristic.
  
ACFA falls in the class of simple unstable theories that was identified by S. Shelah \cite{sh1, sh2}.
In this context, non-forking has all the usual properties,
but unlike in stable theories, stationarity fails.
A good substitute can be found in the Independence Theorem,
which can be used to combine types and allows one to generalise the notions of generic type of a group and of stabilisers of types to groups definable in models of ACFA.
B. Kim and A. Pillay showed in \cite{KP} that a first order theory is simple if and only if it has a notion of independence that has the usual properties of non-forking and satisfies the Independence Theorem.

In \cite{ChHr}, Chatzidakis and Hrushovski define an independence notion in models of ACFA by letting,  for any set $A$, the model $acl_\sigma(A)$
be the field that is obtained by closing the field generated by the set $A$ with respect to the distinguished automorphism and its inverse, and then taking the (field theoretic) algebraic closure.
They then define $A$ to be independent from $B$ over $C$ if $acl_\sigma(AC)$ 
is algebraically independent from $acl_\sigma(BC)$ over $acl_\sigma(C)$.
This notion inherits the usual properties of non-forking from algebraic independence in fields,
and Chatzidakis and Hrushovski show that it satisfies a more general version of the independence theorem (Generalised Independence Theorem) that allows them to combine any finite number of types simultaneously.
This implies, by \cite{KP}, that the theory of ACFA is simple and the independence relation coincides with the usual notion of non-forking.
 
A natural way to generalise a difference field is to add more distinguished automorphisms,
and in \cite{manin}, Hrushovski works at times in this more general setting.
However, there the geometries become quite wild,
and the topology obtained from zero sets of difference polynomials is no longer Noetherian.
Moreover, the strongest results from \cite{ChHr} do not apply.
Thus, it makes sense to pose some restrictions to the group of distinguished automorphisms in order to get a more well-behaved model class.
A natural idea is to require the automorphisms to commute.
However, here one runs into the problem that the existentially closed models of the theory of fields with several commuting automorphisms do not form a first order model class.
Hrushovski has come up with a counterexample already in the case of two automorphisms  (the proof can be found from \cite{kikyo}).

However, the existentially closed models of the theory of fields with commuting automorphisms \emph{do} form an abstract elementary class (AEC),
and in many instances, geometric stability theory can be developed for AECs.
In the present paper, we introduce \emph{FCA-classes}, an AEC framework for studying fields with commuting automorphisms.
The main difference to difference fields is that existentially closed models of the theory of fields with commuting automorphisms are not necessarily algebraically closed as fields (see Example \ref{quaternion}).
We solve this problem by taking an FCA-class to consist of the relatively algebraically closed models of the theory of fields with commuting automorphisms (see Definition \ref{relalg}).
This class will contain all the existentially closed models, and in particular, it will have an existentially closed monster model.

An FCA-class will have the amalgamation property (AP) (Lemma \ref{AP}) and joint embedding property (JEP) (Lemma \ref{JEP}),
and thus we can work in a $\kappa$-universal and $\kappa$ -model homogeneous monster model for an arbitrary large cardinal $\kappa$.  
Moreover, the class will be homogeneous, in existentially closed models, existential types will coincide with Galois types, and we get a first order characterisation for Galois types in all models (Lemma \ref{etyypitmaar}). 
Following the lines of \cite{ChHr}, we define an independence notion that is based on algebraic independence in fields
(Definition \ref{independence}).
We then show that it has the properties of non-forking that one would expect in a simple unstable setting (Lemma \ref{nonforkprop}).

Moreover, we prove a version of the independence theorem that allows us to combine three types if they satisfy certain conditions (Theorem \ref{ind}).  
One of the main differences to \cite{ChHr}
is that we need to require that one of these types has certain technical qualities which we call \emph{nice} (see Definition \ref{nice}).  
We show that any type has a free extension that is nice (Lemma \ref{laajenee}).
Thus, when examining a specific type -- which is often the case in geometric stability theory --
we can always replace it with a nice one by extending the base.
This is in  line with other results on stability theory in non-elementary settings where it is not usually possible to consider types over arbitrary sets but they need to be extended to ``rich" enough models. 

Simplicity is often a minimum requirement for geometric stability theory. 
S. Buechler and O. Lessman introduced a notion of simplicity in one non-elementary context,
that of strongly homogeneous structures \cite{bl}.
A monster model for an FCA-class is strongly homogeneous in the sense of \cite{bl},
and we will prove that it is also simple in their sense (Corollary \ref{SIMPLE}).
Buechler and Lessman show that a simple strongly homogeneous model satisfies a version of the Independence Theorem which they call the type amalgamation theorem (Theorem 3.8 in \cite{bl}).
Adapted to our context, it states that if $c$, $b_i$, and $a_i$ are finite tuples such that
\begin{enumerate}[(1)]
\item $b_1$ is independent from $b_2$ over $c$,
\item $a_1$ and $a_2$ have the same Lascar strong type over $c$ (i.e. they are equivalent in every $c$-invariant equivalence relation that has only boundedly many equivalence classes),
\item $a_i$ is independent from $b_i$ over $c$, 
\end{enumerate} 

then there is some $a$ that realizes the Lascar strong type of $a_i$ over $cb_i$ for $i=1,2$ 
and is independent from $b_1b_2$ over $c$.
This will, then, also be true of a monster model of an FCA-class.  
It follows that we will be able to use our independence theorem (Theorem \ref{ind}) without the niceness assumption if we know that one of the types $p_{ij}(x_i,x_j) \in S(\A)$ implies that $x_i$ and $x_j$ have the same Lascar type over $\A$.

The present paper is a beginning of the stability theoretic study of fields with commuting automorphisms in an AEC framework.
In the future, we aim to investigate whether some of the results that hold in ACFA could be generalised to our context.
    
This paper is structured as follows.
In section 2, we take a detour to a more general framework of multiuniversal AECs and study existential types there. We show that under certain conditions,
there is a first order characterisation for Galois types, existential types determine Galois types, and the class is homogeneous. 
Moreover, under additional assumptions of AP and JEP, 
first order types will coincide with Galois types in existentially closed models.

In section 3, we turn our attention to fields with commuting automorphisms,
introduce our AEC framework, prove some of its basic properties and point out that it is a multiuniversal AEC satisfying the requirements given in section 2.
It then follows that in existentially closed models, Galois types are the same as existential types and the class is homogeneous.
In section 4, we present our independence notion, and show that it has the usual properties of non-forking and prove our version of the independence theorem. 
Finally, in section 5, we use these results to show that the class is simple in the sense of \cite{bl}. 
  
We denote the field theoretic algebraic closure of $K$ by $K^{alg}$.
Moreover, we follow the usual model theoretic convention and write $AB$ for $A \cup B$
and $a \in A$ to denote that $a$ is a tuple of elements from $A$. 
Most of the time we use this notation for finite tuples,
and when we deal with infinite tuples, we mention it specifically. 

\section{A remark on existential types in multiuniversal classes}

\noindent  
In this section, we work in a more general framework than in the rest of the paper and take a look at existential types in multiuniversal classes.
In this context, a closure operation can be defined inside a model by taking the smallest strong submodel that contains a given set (see Definition \ref{multiuni}).
We will show that if there exists a collection of quantifier free formulae $\E$ which determines the closure and satisfies some additional requirements,
then Galois types coincide with types that are determined by a collection $\E^+$ of first order formulae (Lemma \ref{etyypit}).  
Moreover, if the class has AP and JEP, then it is homogeneous (Corollary \ref{egalois}) and
in existentially closed models, Galois types will be the same as existential types  (Corollary \ref{multiremark}).
The abstract elementary class that we present in the next section as a framework for studying fields with commuting automorphisms will satisfy these assumptions.
 A reader who so wishes can skip this section and go straight to the next one where we start working in the more specific setting of fields with commuting automorphisms.

In \cite{multi} (Theorem 3.3), it is proved that in multiuniversal classes, 
Galois types of infinite sequences are determined by the Galois types of finite subsequences,
and this implies that a multiuniversal class with AP and JEP is homogeneous. 
However, our result (Lemma \ref{etyypit}) will imply Theorem 3.3. in \cite{multi},
as we will point out in Remark \ref{multire}.
In \cite{pi} (Lemma 2.10),
it is shown that in a suitable subclass of an essentially $\forall$-definable class,
existential types coincide with automorphism types in sufficiently rich models.
However, automorphism types do not necessarily imply Galois types unless the class has AP and JEP. 
In Example \ref{gnote},
we will present a multiuniversal and $\forall$-definable AEC where amalgamation fails and Galois types do not coincide with existential types even in existentially closed models.

In the setting of \cite{pi},
Galois types are the same as existential types if the class has AP and JEP and there are arbitrarily large rich structures, but it is not evident that such structures can be always found. 
In contrast to this, our framework gives a characterisation for Galois types in all models. 
Moreover, we get both homogeneity and the characterisation of Galois types as result of the same proof.
Furthermore, \cite{pi} implicitly makes some cardinal assumptions (see the discussion on p. 4 there, just before Definition 2.7).
It is possible to get rid of those assumptions by modifying the proof, but it requires some extra effort.

We now recall some basic notions related to AECs.
We will use the terms \emph{model} and \emph{structure} interchangeably.
For basic terminology in model theory, see \cite{marker}.

\begin{definition} 
Let $\mathcal{L}$ be a  countable language, let 
$\K$ be a class of $\mathcal{L}$-structures and let $\preccurlyeq$ be a binary relation on $\K$. 
We say $(\K, \preccurlyeq)$ is an abstract elementary class (AEC for short) if the following hold.
\begin{enumerate}[(1)]
\item Both $\K$ and $\preccurlyeq$ are closed under isomorphisms.
\item If $\A, \B \in \K$ and $\A \preccurlyeq \B$, then $\A \subseteq \B$.
\item The relation $\preccurlyeq$ is a partial order on $\K$. 
\item If $\delta$ is a cardinal and $(\A_i \, | \, i<\delta)$
is a $\preccurlyeq$-increasing chain of structures in $\K$, then
\begin{enumerate}[a)]
\item $\bigcup_{i<\delta} \A_i \in \K$;
\item for each $j<\delta$, $\A_j \preccurlyeq \bigcup_{i<\delta} \A_i$;
\item if $\B \in \K$ and for each $i<\delta$, $\A_i \preccurlyeq \B$, 
then $\bigcup_{i<\delta} \A_i \preccurlyeq \B$.
\end{enumerate}
\item If $\A, \B, \C \in \K$, $\A \preccurlyeq \C$, $\B \preccurlyeq \C$ and $\A \subseteq \B$,
then $\A \preccurlyeq \B$.
\item There is a L\"owenheim-Skolem number $LS(\K)$ such that if 
$\A \in \K$ and $\B \subseteq \A$,
then there is some structure $\A' \in \K$
such that $\B \subseteq \A' \preccurlyeq \A$ and 
$\vert \A' \vert = \vert \B \vert + LS(\K)$.
\end{enumerate}

If $\A \preccurlyeq \B$, we say $\A$ is a \emph{strong substructure} of $\B$. 
\end{definition}

\begin{definition}
Let $(\K, \preccurlyeq)$ be an AEC, and let $\A, \B \in \K$.
We say a map $f: \A \to \B$ is a \emph{strong embedding} if $f: \A \to f(\A)$
is an isomorphism and $f(\A) \preccurlyeq \B$. 
\end{definition}

\begin{definition}
Let $(\K, \preccurlyeq)$ be an AEC, and let
$\k$ be a cardinal.
A model $\M \in \K$ is \emph{$\k$-universal} if for every $\A \in \K$
such that $\vert \A \vert <\k$,
there is a strong embedding $f: \A \to \M$.
\end{definition}

\begin{definition}\label{modhom}
Let $\k$ be a cardinal.
A model $\M \in \K$ is \emph{$\k$- model homogeneous} if whenever $\A, \B \preccurlyeq \M$
are such that $\vert \A \vert, \vert \B \vert < \k$ and $f: \A \to \B$
is an isomorphism, then $f$ extends to an automorphism of $\M$.
\end{definition}

\begin{definition}\label{ap}
We say an AEC $\K$ has the \emph{amalgamation property}
if for any $\A, \B, \C \in  \K$ such that there are strong embeddings $f: \C \to \A$
and $f': \C \to \B$, 
there is some $\D \in \K$ and strong embeddings $g: \A \to \D$
and $g': \B \to \D$ such that $g \circ f =g' \circ f'$.

Moreover, if $\D$ and the embeddings $g$ and $g'$
can be chosen so that $g(\A) \cap g'(\B)=g(f(\C))$,
then we say $\K$ has the \emph{disjoint amalgamation property}. 
\end{definition}

\begin{definition}\label{jep}
We say an AEC $\K$ has the \emph{joint embedding property}
if for any $\A, \B \in \K$, there is some $\C \in \K$
and strong embeddings $f: \A \to \C$, $g: \B \to \C$.
\end{definition}

It is well known that if $(\K, \preccurlyeq)$ is an AEC with AP and JEP, then it contains, 
for each cardinal $\k$,
a $\k$-universal and $\k$- model homogeneous model (see e.g. \cite{baldwin}, Exercise 8.6).
In a context like that,
it is practical to work inside such a model  for some large $\k$,
and we call it a \emph{monster model} for $\K$.
This will be the case with the AEC framework that we will present in the next section. 
If $\A, \B \in \K$ and $\A \preccurlyeq \B$, we say $\A$ is a \emph{strong submodel} of $\B$.
We now recall the notion of Galois types.  
If a monster model exists, then Galois types will become orbits of automorphisms of the monster.

Since our goal is to give an alternative proof to a result from \cite{multi},
we want to give the exact same definition for Galois types that they use (Definition 2.16 in \cite{vasey}).
Note that the below definition requires $A \subseteq \C$. 
 
\begin{definition}\label{Eat}
Let $\K$ be an AEC, and let $\K_3$ be the set of triplets $(\bar{a}, A, \A)$,
where $\A \in \K$, $A \subseteq \A$, and $\bar{a}$ is a (possibly infinite) sequence of elements from $\A$.
We define Galois types as follows.
\begin{itemize}
\item If $(\bar{a}, A, \A), (\bar{b}, B, \B) \in \K_3$, we define the relation $E_{at}$ so that $(\bar{a}, A, \A)E_{at}(\bar{b}, B, \B)$ if $A=B$, and there exist some $\C \in \K$ 
and strong embeddings $f_1: \A \to \C$ and $f_2: \B \to \C$ such that $f_1 \raj A= f_2 \raj A=id$ and $f_1(\bar{a})=f_2(\bar{b})$.
\item We let $E$ be the transitive closure of $E_{at}$ (note that $E$ is an equivalence relation).
\item For $(\bar{a}, A, \A) \in \K_3$, we let the \emph{Galois type} of $\bar{a}$ over $A$ in $\A$,
denoted $tp^g(\bar{a}/A; \A)$ to be the $E$- equivalence class of $(\bar{a}, A, \A)$.
\end{itemize}
 
\end{definition}

If the model $\A$ is clear from the context (e.g. when we will be working in a universal, model homogeneous monster model), we will write just $tp^g(\bar{a}/A)$ for $tp^g(\bar{a}/A; \A)$.

\begin{definition}
Let $\Phi$ be a collection of first order formulae, 
let $\A$ be a model, let $A \subseteq \A$,
and let $b \in A$ be a finite tuple.
We define the $\Phi$-type of $b$ over $A$ in $\A$,
denoted $tp_\Phi^\A(b/A)$,
to be the set of all formulae $\phi(x,a)$ where either $\phi \in \Phi$
or $\neg \phi \in \Phi$, and
$a \in A$ is a finite tuple,
such that $\A \models \phi(b,a)$.
\end{definition}
 
\begin{definition}\label{etyypmaar}
Let $\Phi$ be a collection of first order formulae.
We define a $\Phi$-$n$-type over a set $A$ to be a collection $\Sigma$
of formulae $\phi(x_1, \ldots, x_n, a)$, where $\phi$ or $\neg \phi$ is in $\Phi$,
and $a$ is a tuple from $A$,  
such that for some model $\A$ containing $A$ and some $n$-tuple $b \in \A$,
$\Sigma=tp_\Phi^\A(b/A)$.
 
If $\Phi$ is taken to be the collection of all formulae of the form $\exists x_1 \cdots \exists x_n \phi(x_1, \ldots, x_n, \bar{y})$, where $\phi$ is quantifier free, then we call $\Phi$-types \emph{existential types}
and use $tp_\exists^\A(a/A)$ to denote the existential type of $a$ over $A$ in $\A$.
\end{definition}
 
We now recall the definition of a \emph{multiuniversal class}, introduced in \cite{multi} (Definition 2.8).

\begin{definition}\label{multiuni}
Suppose $(\K, \preccurlyeq)$ is an AEC.
For any $A \subseteq \A \in \K$,
we define 
$$cl_\A(A)= \bigcap\{\B \, | \, \B \preccurlyeq \A \textrm{ and } A \subseteq \B\}.$$
\vspace{0.2cm}
We say that $\K$ is a \emph{multiuniversal class} if the following hold:
\begin{enumerate}[(i)]
\item For all $\A \in \K$ and $A \subseteq \A$, it holds that $cl_\A(A) \in \K$ and
$cl_\A(A) \preccurlyeq \A$;
\item If $\A \in \K$, $A \subseteq \A$, and $a \in cl_\A(A)$,
then $tp^g(a/A; \A)$ is algebraic (i.e. has only finitely many realisations in $\A$).
\end{enumerate}
\end{definition}

\begin{lemma}\label{closisom}
Let $\K$ be a multiuniversal AEC, $\A, \B \in \K$,
and let $\bar{a} \in \A$, $\bar{b} \in \B$ be possibly infinite sequences such that $tp^g(\bar{a}/\emptyset; \A)=tp^g(\bar{b}/\emptyset; \B)$.
Then, there is an isomorphism $f: cl_\A(\bar{a}) \cong cl_\B(\bar{b})$ such that $f(\bar{a})=\bar{b}$.
\end{lemma}

\begin{proof}
It suffices to prove the lemma in case $(\bar{a}, \emptyset, \A)E_{at}(\bar{b}, \emptyset, \B)$.
Then, there is some $\C \in \K$ and strong embeddings $f_1: \A \to \C$ and $f_2: \B \to \C$
such that $\bar{c}=f_1(\bar{a})=f_2(\bar{b}).$
Since $cl_\A(\bar{a})$ is the smallest strong submodel of $\A$ containing $\bar{a}$,
we have $f_1(cl_\A(\bar{a}))=cl_\C(\bar{c})$ by transitivity of $\preccurlyeq$.
Likewise, $f_2(cl_\B(\bar{b}))=cl_\C(\bar{c})$, 
and $f_2^{-1} \circ f_1$ gives the desired isomorphism.
\end{proof}

Next, we illustrate with an example that if amalgamation fails in a
$\forall$-definable multiuniversal AEC, then Galois types do not necessarily coincide with existential types even in existentially closed models.

\begin{definition}
Let $(\K, \preccurlyeq)$ be an AEC. 
We say a model $\A \in \K$ is \emph{existentially closed} if the following condition holds:
if $\B \in \K$, $\A \preccurlyeq \B$, $a \in \A$, $\phi$ is a quantifier free formula, and
$\B \models \exists x_1 \cdots \exists x_n \phi(x_1, \ldots, x_n,a)$,  
then $\A \models \exists x_1 \cdots \exists x_n \phi(x_1, \ldots, x_n,a)$.
\end{definition}

\begin{example}\label{gnote}
Let $L=\{P, Q, R, c\}$,
where $P$ is a unary predicate, $R$ and $Q$ are binary predicates, 
and $c$ is a constant symbol.
Let $T$ be a theory that states the following:
\begin{itemize}
\item $c \notin P$;
\item If $(x,y) \in R$, then $x \in P$, $y \notin P$, and $y \neq c$;
\item If $(x,y) \in Q$, then $x, y \notin P$, $x \neq y$, $x \neq c$, and  $y \neq c$;
\item For any $x \in P$, there are at most two $y$ such that $R(x,y)$;
\item If $x, y \in P$ and $x \neq y$, then there is no $z$ such that $R(x,z) \wedge R(y,z)$.
\end{itemize}
Let $(\K, \subseteq)$ be the AEC that consists of models of $T$, 
equipped with the submodel relation.
The theory $T$ is clearly $\forall$-axiomatizable,
and $\K$ is a multiuniversal class with $cl_\A(A)=A \cup \{c^\A\}$. 

The class $\K$ does not have amalgamation.
Indeed, let $\A \in \K$ be a model and let $a \in P^\A$
be such that there is no $b \in \A$ with $(a,b) \in R^\A$.
Then, there are $\B, \C \in \K$ and $b_1, b_2 \in \B$, $c_1, c_2 \in \C$ such that $\A \subseteq \B$, $\A \subseteq \C$, $(a, b_1), (a, b_2) \in R^\B$, $(b_1, b_2) \in Q^\B$,
$(a, c_1), (a, c_2) \in R^\C$ and $(c_1,c_2) \notin Q^\C$.
Now, $\B$ and $\C$ cannot be amalgamated over $\A$.

In $\K$, Galois types are not the same as existential types,
not even in existentially closed models.
Indeed, let $\A \in \K$ be existentially closed, and suppose $a_1, a_2, a_3, b_1, b_2, b_3 \in \A$
are such that  $(a_1, a_2), (a_1,a_3), (b_1,b_2), (b_1,b_3) \in R^\A$, 
$(a_2, a_3) \in Q^\A$, and $(b_2, b_3) \notin Q^\A$.
Now, $tp_\exists^\A(a_1/\emptyset) \neq tp_\exists^\A(b_1/\emptyset)$
since $\A \models \exists x \exists y (R(a_1, x) \wedge R(a_1,y) \wedge Q(x,y))$
and $\A \not\models  \exists x \exists y (R(b_1, x) \wedge R(b_1, y) \wedge Q(x,y))$.
On the other hand,
$tp^g(a_1/\emptyset; \A)=tp^g(b_1/\emptyset; \A)$.
Indeed, let $\B \in \K$ be such that $\B=\{c^\B, d\}$, where $P^\B=\{d\}$.
There are strong embeddings $f: \B \to \A$ and $g: \B\to \A$
with $f(c^\B)=g(c^\B)=c^\A$, $f(d)=a_1$ and $g(d)=b_1$,
so $(a_1, \emptyset, \A) E_{at} (d, \emptyset, \B)$,
and $(d, \emptyset, \B) E_{at} (b_2, \emptyset, \A)$.
\end{example}

We will prove that existential types imply Galois types in the case that the closure operation from Definition \ref{multiuni} is obtained from a collection $\E$ of quantifier free first order formulae that satisfy certain requirements,
as explained in the following definition.
 
 \begin{definition}\label{eclos}
Let $\mathcal{E}$ be a collection of quantifier free formulae of the form $\phi(x, \bar{y})$,
where $x$ is a single variable and $\bar{y}$ a tuple of variables,
such that whenever $\A \in \K$ and $\bar{b} \in \A$, then $\phi(\A, \bar{b})$ is finite.
Moreover, suppose the formula $x=y$ is in $\E$.

If $\A \in \K$ and $A \subseteq \A$, 
we define the \emph{$\E$-closure} of $A$ in $\A$, denoted $\E$-$cl_\A(A)$,
to be the set of all elements $a \in \A$
such that there is a formula $\phi(x, \bar{y}) \in \E$ and a finite tuple $\bar{b} \in A$ such that
$\A \models \phi(a, \bar{b})$.
\end{definition}

In our main results, we will assume we are working in a multiuniversal class and $\E$-$cl$ equals the closure operator $cl$ from Definition \ref{multiuni}.

Note that $\E$-cl depends on $\E$, and the requirements for the collection $\E$ are rather strong: if $\phi(x, \bar{y}) \in \E$,
then the set $\phi(\A, \bar{b})$ must be finite for all $\A \in \K$ and $\bar{b} \in \A$.
Thus, for example, the formula $\neg x =y$ cannot be in $\E$ if $\K$ contains infinite models.

Next, we will define a collection $\E^+$ of first order formulae such that $\E \subseteq \E^+$.
We will eventually show that if we are working in a multiuniversal AEC and $\E$-closure equals the closure operation from Definition \ref{multiuni},
then Galois types will coincide with $\E^+$-types (see below).
After that, we will point out that existential types imply $\E^+$-types and thus Galois types. 

If $(\K, \preccurlyeq)$ is as in Example \ref{gnote},
and  we take $\E$ to consist of formulae of the form $x=y$
and $x=c$, then, 
for any $\A \in \K$ and $A \subseteq \A$, we have $cl_\A(A)=\E$-$cl_\A(A)$.

\begin{definition}\label{eclos2}
Let $x$ be a single variable and use $\exists^{=n}x\phi(x)$ as shorthand for the formula stating that there are exactly $n$
many $x$ such that $\phi(x)$, and define the collection of $\E^+$ to consist exactly of the formulae that we can construct recursively as follows (i.e. (ii) and (iii) can be iterated):
\begin{enumerate}[(i)]
\item If $\phi$ is an atomic formula or a negated atomic formula, then $\phi \in \E^+$;
\item If $\phi, \psi \in \E^+$, then $\phi \wedge \psi, \phi \vee \psi \in \E^+$;
\item If $m,n,k \in \o \setminus \{0\}$, $\phi_i(x, \bar{y}_i) \in \E^+$ for $i=1, \ldots, m$, and $\phi_1(x, \bar{y}_1) \in \E$,
then  $\exists^{=n} x \bigwedge_{i=1}^m \phi_i(x,\bar{y_i}) \wedge \exists^{=k} x\phi_1(x,\bar{y}_1) \in \E^+$. 
\end{enumerate}
\end{definition}

Note that $\E^+$ contains all quantifier free formulae, and thus $\E \subseteq \E^+$.
Note also that when we build the formulae using the recipe from the definition,
we are only allowed to quantify over the (single) variable $x$, and the variables $\bar{y}$ will remain free.
   
The motivation behind the collection $\E^+$ is to define a set of formulae such that in a multiuniversal AEC, under the assumption that $\E$-$cl$ equals the closure operator from Definition \ref{multiuni}, 
$tp_{\E^+}^\A(\bar{a}/\emptyset)$ defines the isomorphism type over the (possibly infinite) tuple $\bar{a} \in \A$ of the set which you get when you close $\bar{a}$ under the operation
$\E$-$cl_\A$ (i.e. if $tp_{\E^+}^\A(\bar{a}/\emptyset)=tp_{\E^+}^\A(\bar{b}/\emptyset)$, then $\E$-$cl_\A (\bar{a})$ will be isomorphic with $\E$-$cl_\A (\bar{b})$).
In the proof of Lemma \ref{etyypit} we will see that $\E$-$cl_\A$ actually works this way.
 
Note that $tp_{\E^+}^\A(\bar{a}/\emptyset)$ depends on the collection $\E$ since Definition \ref{eclos2}, (iii) requires that $\phi_1 \in \E$.
The collection $\E^+$ will not, generally, contain all existential formulae, but only those needed for the description.
Note also that $\E$-$cl$ will not generally be a closure operator but we are interested in cases where it coincides with the closure operatorfrom Definition \ref{multiuni}.

In Example \ref{gnote},
$\E^+$-types are equivalent to quantifier free types in a given $\A \in \K$ if we take, again, $\E$ to consist of formulae of the form $x=y$ and $x=c$.
Indeed, if $\psi(\bar{y})$ is a formula constructed as in (iii) of Definition \ref{eclos2}
and there is some $\A \in \K$ and $b \in \A$
such that $\A \models \psi(b)$,
then $\phi_1(x, y)$ must be of the form $x=y$ or $x=c$ and $n=k=1$.
It then follows inductively that all satisfiable formulae are equivalent to quantifier free formulae
(in the sense that for each formula $\phi$ that is satisfiable in some $\A \in \K$, 
there exists a quantifier free formula $\psi$ such that for any $\A \in \K$ and $a \in \A$,
$\A \models \phi(a)$ if and only if $\A \models \psi(a)$).

As another example, consider the structure $\mathfrak{M}=(\mathbb{Q}, R)$, where $R^\mathfrak{M}(x,y)$ if $y=x+1$.
If $A \subseteq \mathbb{Q}$, we define $cl(A)=\{x+z \, | \, x \in A, z \in \mathbb{Z}$\}.
Take $\E$ to consist of formulae of the form $x=y$ and $R(x,y)$.
Now, Definition \ref{eclos2} gives a way to construct recursively the formulae that prove $a \in cl(A)$.
For example, if $a, b \in \mathbb{Q}$, then we can write a formula $\phi(x,y) \in \E^+$ such that
$b=a+3$ if and only if $\mathfrak{M} \models \phi(a,b)$.
Indeed, let 
$$\phi(x,y):=\exists^{=1}z(R(x,z) \wedge \psi(y,z)) \wedge \exists^{=1}zR(x,z),$$
where
$$\psi(y,z):=\exists^{=1}w(R(z,w) \wedge R(w,y)) \wedge \exists^{=1}wR(z,w).$$

The way we defined $\E$ and $\E$-$cl$ in Definitions \ref{eclos} and \ref{eclos2} might not feel natural to some model theorists.
In the following remark, we present an alternative way to define $\E$ and $\E$-$cl$ which is more in line with the usual conventions in model theory.
However, we chose to use Definitions \ref{eclos} and \ref{eclos2} since we have found them more practical. 
Nevertheless, the same results can be obtained with the following alternative definition.

\begin{remark}\label{alternative}
In Definition \ref{eclos},
we required that if $\phi(x, \bar{y}) \in \E$, then
$\phi(\A, \bar{b})$ is finite for any $\A \in \K$ and $\bar{b} \in \A$.
This is a rather strong requirement since we require the set defined by $\phi(x, \bar{y})$ 
to be finite in \emph{all} models in $\K$.
However, we could define $\E$ and $\E$-$cl$ alternatively  as follows.
Let $\mathcal{E}$ be a collection of quantifier free formulae of the form $\phi(x, \bar{y})$,
where $x$ is a single variable and $\bar{y}$ a tuple of variables,
such that  the formula $x=y$ is in $\E$,
and if the formula $\phi(x, \bar{y})$ is in $\E$, then the formula $\psi(x, \bar{y},z):=``\phi(x, \bar{y}) \wedge \neg x=z"$
is in $\E$.
Moreover, we require that if $\A, \B \in \K$ and $\A \preccurlyeq \B$,
then $\A$ is $\E$-existentially closed in $\B$
(i.e. if $\phi(x, \bar{y}) \in \E$, $\bar{a} \in \A$ is a finite tuple, 
and $b \in \B$ is such that $\B \models \phi(b,\bar{a})$,
then there is some $c \in \A$ such that $\A \models \phi(c,\bar{a})$).

If $\A \in \K$ and $A \subseteq \A$, 
we define $\E$-$cl_\A(A)$
to be the set of all elements $a \in \A$
such that there is a formula $\phi(x, \bar{y}) \in \E$ and a finite tuple $\bar{b} \in A$ such that
$\A \models \phi(a, \bar{b})$ and $\vert \phi(\A, \bar{b}) \vert$ is finite.
Define then the collection $\E^+$ exactly as in Definition \ref{eclos2}.

Consider formulae that are as in Definition \ref{eclos2} (iii).
We will see that if we define $\E$ as in Definition \ref{eclos}, 
then we can actually omit the last conjunct ($ \exists^{=k} x\phi_1(x,\bar{y}_1)$) in those formulae
(it is not needed in the proofs in the rest of this section).
That is, we can assume that item (iii) in Definition \ref{eclos2} stands as follows:
\begin{enumerate}[(iii)]
\item If $m,n \in \o \setminus \{0\}$, $\phi_i(x, \bar{y}_i) \in \E^+$ for $i=1, \ldots, m$, and $\phi_1(x, \bar{y}_1) \in \E$,
then  $\exists^{=n} x \bigwedge_{i=1}^m \phi_i(x,\bar{y_i}) \in \E^+$. 
\end{enumerate}
However, with the alternative definition described above,  
we will need that last conjunct when we prove that Galois types imply $\E^+$-types
(see the proof of Lemma \ref{etyypit} and Remark \ref{toinenmaar}). 
\end{remark}

From now on, we will assume we are working in a multiuniversal AEC $(\K, \preccurlyeq)$
and that the collections $\E$ and $\E^+$ are as in Definitions \ref{eclos} and \ref{eclos2}.
We will repeat these assumptions in statements of theorems and key lemmas.
 
\begin{definition}
Let $\A \in \K$ and let $\phi(\bar{x})$ be a formula with parameters from $\A$
such that 
$n=\vert \phi(\A) \vert$ is finite.
Let $\mathcal{F}$ be a collection of formulae with free variables $\bar{x}$,
with parameters from $\A$.
If $b \in \A$ is a tuple, $\A \models \phi(b)$
and $\vert \psi(\A) \vert \ge n$
for all formulae $\psi(\bar{x}) \in \mathcal{F}$  with
$\A \models \psi(b)$,
then we say that  $b$ is a \emph{generic realisation} of $\phi(\bar{x})$ with respect to $\mathcal{F}$ in $\A$.

Unless mentioned otherwise, we will assume $\mathcal{F}$
consists of all the formulae from $\E^+$
with parameters from some set $A \subseteq \A$.
We then say $b$ is \emph{generic over $A$ in $\A$}.
\end{definition}    

Our proof that Galois types agree with $\E^+$-types if the two closures coincide will be based on the fact that for each singleton
$a \in \E$-$cl_\A(A)$,
there is some formula $\phi(y) \in \E^+$ with parameters from $A$ such that $a$ is a generic realisation of $\phi$ in $\A$
with respect to the collection of $\E^+$-formulae with parameters from $A$. 
The formula  $\phi$ then determines the $\E^+$-type of $a$ in $\A$.

\begin{lemma}\label{lisalemma}
Let $\A$ be a model, let $\phi(\bar{x})$ be a formula with parameters in $\A$, and let $\mathcal{F}$ 
be a set of formulae in the free variables $\bar{x}$ and with parameters in $\A$.
Assume that $\phi(\A)$ is finite and contains the tuple $\bar{b}$
and there is some formula $\theta(\bar{x}) \in \mathcal{F}$ such that $\A \models \theta(\bar{b})$.
Then there is a formula $\psi \in \mathcal{F}$ such that $\A \models \psi(\bar{b})$ 
and whenever $\theta(\bar{x})$ is satisfied by $\bar{b}$ in $\A$,
then $|(\phi \wedge \psi)(\A)| \le | (\phi \wedge \theta)(\A)|$.

In particular, if $\phi(x) \in \mathcal{E}$ is a formula with parameters in $A \subseteq \A$, and $b \in \phi(\A)$,
then there is an $\mathcal{L}(A)$-formula $\psi(x)$ in $\mathcal{E}^+$ such that $b$ is a generic of
$\phi \wedge \psi$ over $A$ in $\A$.
\end{lemma}

\begin{proof}
This follows immediately from the fact that the infimum of
$$\{| (\phi \wedge \psi)(\A)| \, | \, \A \models \psi(\bar{b}), \psi \in \mathcal{F}\}$$
is attained for some $\psi \in \mathcal{F}$.

\end{proof}
   
\begin{lemma}\label{gensamat}
Suppose $\A$ and $\B$ are models,
$\bar{a} \in \A$, $\bar{b} \in \B$, are two possibly infinite sequences,
$\phi(x,y) \in \E^+$,
$c \in \E$-$cl_\A(\bar{a})$ is a generic realisation of $\phi(x,\bar{a})$ over $\bar{a}$ in $\A$,
and $tp_{\E^{+}}^\A(\bar{a}/\emptyset)=tp_{\E^{+}}^\B(\bar{b}/\emptyset)$.
Then, there is some $d \in \E$-$cl_\B(\bar{b})$
such that $tp_{\E^{+}}^\A(\bar{a},c/\emptyset)=tp_{\E^{+}}^\B(\bar{b},d/\emptyset)$.
\end{lemma}

\begin{proof}
Write $\bar{a}=(a_i)_{i<\a}$, $\bar{b}=(b_i)_{i<\alpha}$, and let
$g: \bar{a} \to \bar{b}$ be such that $g(a_i)=b_i$ for $i<\a$.
Since $c \in\E$-$cl_\A(\bar{a})$,
there is some formula $\psi(x,y) \in \E$ 
such that $\A \models \psi(c, \bar{a})$. 
As $c$ is a generic realisation of $\phi(x,\bar{a})$ over $\bar{a}$,
we have $\phi(\A, \bar{a})=\phi(\A,\bar{a}) \wedge \psi(\A,\bar{a})$.
We will find an element $d \in \B$ which is a generic realisation of
$\phi(x, \bar{b}) \wedge \psi(x,\bar{b})$ over $\bar{b}$ in $\B$,
and then show that $d$ is as wanted.

Suppose, for the sake of contradiction, that  no element in $(\psi \wedge \phi)(\B, \bar{b})$ is generic over $\bar{b}$.
Let $n=\vert (\psi \wedge \phi) (\A, \bar{a})\vert$ and $m= \vert \psi(\A, \bar{a}) \vert$.
The formula 
$$\Psi(y)=\exists^{=n} x (\psi(x,y) \wedge \phi(x,y)) \wedge \exists^{=m}x\psi(x,y)$$
is in $\E^+$, $tp_{\E^{+}}^\A(\bar{a}/\emptyset)=tp_{\E^{+}}^\B(\bar{b}/\emptyset)$,
and $\A \models\Psi(\bar{a})$,
so $\vert (\psi \wedge \phi )(\B, \bar{b}) \vert=n$.
 
Write $(\psi \wedge \phi)(\B, \bar{b})=\{d_1, \ldots, d_n\}$.
By the counterassumption, there is, for each $d_i$,
some formula $\chi_i(x, z) \in \E^+$
such that $\B \models \chi_i(d_i, \bar{b})$ and $\vert \chi_i(\B, \bar{b}) \vert=n_i <n$.
Denoting
$$\Phi(x,y,z)=\psi(x,y) \wedge \phi(x,y) \wedge \bigvee_{i=1}^n \chi_i(x,z),$$
we have $\exists^{=n}x \Phi(x,y,z) \wedge \exists^{=m}x \psi(x,y) \in \E^+$.
Since $tp_{\E^{+}}^\A(\bar{a}/\emptyset)=tp_{\E^{+}}^\B(\bar{b}/\emptyset),$
and there are exactly $n$ many distinct $d \in \B$
such that $\B \models \Phi(d, \bar{b})$,
we have $\vert \Phi(\A, \bar{a}) \vert=n$,
and thus $\Phi(\A, \bar{a})=(\phi \wedge \psi)(\A, \bar{a})$.
On the other hand, 
$\exists^{=n_i} x (\psi \wedge \chi_i)(x,y,z) \wedge \exists^{=m}x\psi(x,y) \in \E^+$,
and thus $\vert \psi(\A, \bar{a}) \wedge \chi_i(\A,\bar{a}) \vert=n_i<n$ for each $i<n$, 
which contradicts the genericity of $c$.

Thus, there is some $d \in \B$ which is a generic realisation of
$\phi(x, \bar{b}) \wedge \psi(x,\bar{b})$ over $\bar{b}$ in $\B$. 
Clearly $d \in \E$-$cl_\B(\bar{b})$.
We show that if $\chi(x,y) \in \E^+$,
and $\A \models \chi(c, \bar{a})$, then $\B \models \chi(d, \bar{b})$ (the other direction is symmetric).
Since $c$ is a generic realization of $\phi(x,\bar{a}) \wedge \psi(x,\bar{a})$ over $\bar{a}$ and $tp_{\E^+}^\A(\bar{a}/\emptyset)=tp_{\E^+}^\B(\bar{b}/\emptyset)$, we have
\begin{eqnarray*}
\vert \phi(\A, \bar{a}) \cap \psi(\A, \bar{a}) \cap \chi(\A,\bar{a}) \vert&=& \vert \phi(\A,\bar{a}) \vert=n \\
&=& \vert \phi(\B,\bar{b}) \vert=\vert \phi(\A, \bar{a}) \cap \psi(\A, \bar{a}) \cap \chi(\A,\bar{a}) \vert,
\end{eqnarray*}
and thus any realization of $\phi(x,\bar{b})$ satisfies $\psi(x,\bar{b}) \wedge \chi(x,\bar{b}).$
 
\end{proof}

\begin{lemma}\label{etyypit} 
Suppose $\K$ is a multiuniversal AEC,
and there is a collection $\E$ of quantifier free formulae such that for all $\A \in \K$ and $A \subseteq \A$,
$cl_\A(A)=\E$-$cl_\A(A)$.
Then, for any $\A, \B \in \K$, and any (possibly infinite) sequences $\bar{a} \in \A$, $\bar{b} \in \B$,
it holds that $tp^g(\bar{a}/\emptyset; \A)=tp^g(\bar{b}/\emptyset; \B)$ if and only if $tp_{\E^+}^\A(\bar{a}/\emptyset)=tp_{\E^+}^\B(\bar{b}/\emptyset)$.  
\end{lemma}

\begin{proof} 
Suppose first $tp^g(\bar{a}/\emptyset; \A)=tp^g(\bar{b}/\emptyset; \B)$.
Since $cl_\A(\bar{a})=\E$-$cl_\A(\bar{a})$ and $cl_\B(\bar{b})=\E$-$cl_\B(\bar{b})$,
we have  $\E$-$cl_\A(\bar{a}) \preccurlyeq \A$, $\E$-$cl_\B(\bar{b}) \preccurlyeq \B$,
and by Lemma \ref{closisom}, 
there is an isomorphism $f: \E$-$cl_\A(\bar{a}) \to \E$-$cl_\B(\bar{b})$ such that
$f(\bar{a})=\bar{b}$. 

We will prove the claim by showing that if $\phi(x) \in \E^+$ and $c \in \E$-$cl_\A(\bar{a})$ is a finite tuple,
then $\A \models \phi(c)$ if and only if $\B \models \phi(f(c))$.  
We will do this by induction on the formula $\phi$.
If $\phi$ is quantifier free, the claim clearly holds. 
It is also easy to see that if the claim holds for $\phi$ and $\psi$,
then it holds for $\phi \wedge \psi$ and $\phi \vee \psi$.  

Suppose now 
$$\Phi(y)=\exists^{=n}x\bigwedge_{i=1}^m \phi_i(x,y) \wedge \exists^{=k}x \phi_1(x,y),$$
where  $\phi_i \in \E^+$ for each $i$, $\phi_1 \in \E$, and the claim holds for each $\phi_i$.
Assume $\A \models \Phi(c)$.
Let $d_1, \ldots, d_n \in \A$ be the exactly $n$ many distinct elements $d$ such that $\A \models \bigwedge_{i=1}^m \phi_i(d,c)$.
To be able to apply the induction hypothesis, we need to show that $d_i \in \E$-$cl_\A(a)$ for each $i$.
And indeed, we have $\A \models \phi_1(d_i, c)$,
so $d_i \in \E$-$cl_\A(c)$.
Since 
$c \in \E$-$cl_\A(\bar{a})=cl_\A(\bar{a})$, we have 
$$d_i \in \E\textrm{-}cl_\A(c)=cl_\A(c) \subseteq cl_\A(\bar{a})=\E\textrm{-}cl_\A(\bar{a}).$$

Now, by the induction hypothesis, $\B \models \bigwedge_{i=1}^m \phi_i(f(d_j), f(c))$ for $j=1, \ldots,n$.
Moreover, these are all the elements of $\B$ that satisfy this formula.
Indeed, suppose for the sake of contradiction that there is some
$e \in \B$ such that $e \neq  f(d_i)$ for $i=1, \ldots, n$, and
$\B \models \bigwedge_{i=1}^m \phi_i(e, f(c))$.
We have $f(c) \in \E$-$cl_\B(\bar{b})$,
and since $\B \models \phi_1 (e, f(c))$,
it follows that $e \in cl_\B(cl_\B(\bar{b}))=cl_\B(\bar{b})= \E$-$cl_\B(\bar{b})$,
so $f^{-1}$ is defined at $e$.
Thus, $\A \models  \bigwedge_{i=1}^m \phi_i(f^{-1}(e),c)$ 
which contradicts the fact that only $n$ many elements of $\A$ satisfy this formula.
Thus, $\B \models \exists^{=n}x\bigwedge_{i=1}^m \phi_i(x,f(c))$.
Similarly, we can show that $\B \models \exists^{=k}x \phi_1(x,f(c))$.
Hence, $\B \models \Phi(f(c))$,
and we have proved that Galois types imply $\E^+$-types. 
  
For the other direction, 
let $\A, \B \in \K$, and
let $\bar{a}=(a_i)_{i<\a} \in \A$, $\bar{b}=(b_i)_{i<\a} \in \B$,
be two sequences such that
$tp_{\E^+}^\A(\bar{a})=tp_{\E^+}^\B(\bar{b})$.
We will find an isomorphism $f: cl_\A(\bar{a}) \to cl_\B(\bar{b})$
such that $f(\bar{a})=\bar{b}$.
To construct such an isomorphism,
it is enough to show that if $c \in cl_\A(\bar{a}) \setminus \bar{a}$,
then there is some  
$d \in cl_\B(\bar{b}) \setminus \bar{b}$
such that $tp_{\E^+}^\A(\bar{a}, c)=tp_{\E^+}^\B(\bar{b}, d)$. 
This follows from Lemmas \ref{lisalemma} and \ref{gensamat}.

Thus, we have seen that 
$(\bar{a}, \emptyset, cl_\A(\bar{a}))E_{at}(\bar{b}, \emptyset, cl_\B(\bar{b})).$
By multiuniversality, $cl_\A(\bar{a})\preccurlyeq \A$, and thus  
$(\bar{a}, \emptyset, cl_\A(\bar{a}))E_{at}(\bar{a}, \emptyset, \A)$.
Similarly, $(\bar{b}, \emptyset, cl_\B(\bar{b}))E_{at}(\bar{b}, \emptyset, \B)$,
and hence 
$tp^g(\bar{a}/\emptyset; \A)=tp^g(\bar{b}/\emptyset; \B)$. 
\end{proof}

\begin{remark}\label{toinenmaar}
Note that if we define the set $\E$ and the operator $\E$-$cl$ 
as in Remark \ref{alternative},
then we need the last conjunct ($ \exists^{=k} x\phi_1(x,\bar{y}_1)$)
of Definition \ref{eclos2} in the proof of Lemma \ref{etyypit} to conclude that
$d_i \in \E$-$cl_\A(c)$.
Indeed, we get that $\A \models \phi_1(d_i, c)$,
but we also need that $\vert \phi_1(\A,c) \vert$
is finite, and this follows from the fact that the formula $\Phi(y)$
contains $\exists^{=k}x \phi_1(x,y)$ as a conjunct. 
\end{remark}
 
\begin{remark}\label{multire}  
In \cite{multi},
it is shown that if $\K$ is a multiuniversal AEC, 
$\A, \B \in \K$, $\bar{a}=(a_i)_{i<\a} \in \A$,
$\bar{b}=(b_i)_{i<\a} \in \B$ are possibly infinite sequences, and for any finite 
$X \subseteq \a$, it holds that 
$tp^g((a_i)_{i \in X}/\emptyset; \A)=tp^g((b_i)_{i\in X}/\emptyset; \B)$,
then $tp^g(\bar{a}/\emptyset; \A)=tp^g(\bar{b}/\emptyset; \B)$ (Theorem 3.3.).
We note that our Lemma \ref{etyypit} implies this result.

To see this, let $(\K, \preccurlyeq)$ be a multiuniversal AEC,
and let $\E^*$ be the collection of all Galois types $p$ (of arbitrarily long finite tuples)
such that there is some $\A \in \K$,
some finite tuple $\bar{b} \in \A$, and an element $a \in \A$
such that $a \in cl_\A(\bar{b})$ and $a\bar{b}$ realises $p$.
Introduce new relation symbols $R_p$ for $p \in \E^*$,
and let $(\K^*, \preccurlyeq^*)$ be the class that consists of exactly the models of $\K$,
but with extra structure that is given by interpreting the relations $R_p$
so that $R_p(a, \bar{b})$ if and only if $a \bar{b}$ realises $p$.
Define the strong submodel relation $\preccurlyeq_{\K^*}$ in $\K^*$ so that $\A^*\preccurlyeq_{\K^*} \B^*$ 
if and only if $\A \preccurlyeq \B$ holds for the restrictions $\A$ and $\B$ of $\A^*$ and $\B^*$, respectively,
to the original language.
Let $\E$ be the collection of the formulae $R_p(x, \bar{y})$, $p \in \E^*$.

The class $(\K^*, \preccurlyeq_{\K^*})$ is clearly a multiuniversal AEC since $\K$ 
is.
The relations $R_p$ are preserved under isomorphisms since isomorphisms preserve Galois types. 
Thus, strong embeddings stay the same in the new class and preserve the relations $R_p$.

We want to apply Lemma \ref{etyypit}
to show that in this setting, Galois types are the same as existential types.
The result will then follow since existential types of infinite sequences are determined by the existential types of their finite subsets.
To be able to use the lemma, we need to show that for all
$\A \in \K^*$ and $A \subseteq \A$, $\E$-$cl_\A(A)=cl_\A(A)$.
If $a \in cl_\A(A)$,
then there is, by Fact 2.2 in \cite{multi} (or Proposition 2.14 in \cite{vasey2}), some finite $\bar{b} \in A$
such that $a \in cl_\A(\bar{b})$.
Then, $p=tp^g(a, \bar{b}/\emptyset; \A) \in \E^*$,
$\A \models R_p(a, \bar{b})$,
and by condition (ii) of Definition \ref{multiuni},
$R_p(\A, \bar{b})$ is finite,
so $a \in \E$-$cl_\A(\bar{b}) \subseteq \E$-$cl_\A(A).$
 
To see that $\E$-$cl_\A(A)\subseteq cl_\A(A)$,
we need to prove that Galois types determine the closure operation $cl_\A$,
i.e. that if $\A, \B \in \K$,
$a, \bar{b} \in \A$, $c, \bar{d} \in \B$, $a \in cl_\A(\bar{b})$,
and $tp^g(a \bar{b} /\emptyset; \A)=tp^g(c \bar{d} /\emptyset; \B)$,
then $c \in cl_\B(\bar{d})$.
For this, it suffices to show that if $((a, \bar{b}), \emptyset, \A) E_{at} ((c, \bar{d}), \emptyset, \B)$ 
and $a \in cl_\A(\bar{b})$, then $c \in cl_\B(\bar{d})$.
Suppose not.
By the definition of the relation $E_{at}$ (see Definition \ref{Eat}),  
there is some $\C \in \K$ and elementary embeddings
$f: \A \to \C$ and $g: \B \to \C$ such that $f(a)=g(c)$ and $f(\bar{b})=g(\bar{d})$.
Now $f(a) \in cl_\C(f(\bar{b}))$ but $g(c) \notin cl_\C(g(\bar{d}))$,
a contradiction.
\end{remark}

Recall our standing assumptions that we are working in a multiuniversal AEC $(\K, \preccurlyeq)$
and that $\E$ and $\E^+$ are as in Definitions \ref{eclos} and \ref{eclos2}.

\begin{lemma}\label{ee}
Let $\A$, $\B$ be models, and let $a \in \A$, $b \in \B$ be finite tuples.
If $tp_\exists^\A(a/\emptyset)=tp_\exists^\B(b/\emptyset)$,
then $tp_{\E^+}^\A(a/\emptyset)=tp_{\E^+}^\B(b/\emptyset)$
\end{lemma}

\begin{proof}
Suppose $\phi(x) \in \E^+$ and $\A \models \phi(a)$.
We prove $\B \models \phi(b)$ by induction on the formula $\phi$.
The claim holds for quantifier free formulae, and if it holds for
$\phi$ and $\psi$, then it holds for $\neg \phi$, $\phi \wedge \psi$ and $\phi \vee \psi$.

Suppose now
$$\Phi(y)=\exists^{=n}x \bigwedge_{i=1}^m \phi_i(x,y) \wedge \exists^{=k}x\phi_1(x,y),$$
where $\phi_i \in \E^+$ for each $i$, $\phi_1 \in \E$, and the claim holds for each $\phi_i$.
Assume $\A \models \Phi(a)$.
Since $tp_\exists^\A(a/\emptyset)=tp_\exists^\B(b/\emptyset)$,
we have $\vert \phi_1(\A, a) \vert =\vert \phi_1(\B, b) \vert=k$.  Let $\bar{c}=(c_1, \ldots, c_k)$ enumerate the elements in $\phi_1(\A, a)$.
Apply now Lemma \ref{lisalemma} to the formula 
$\psi(\bar{x},a)=\wedge_{i=1}^k \phi_1(x_i,a)$
and to the set of existential formulae in variables $(x_1, \ldots, x_k)$
with parameters in $a$,
to find an existential $\mathcal{L}(a)$-formula  
$\theta(\bar{x}, a)$
such that $\A \models \theta(\bar{c},a)$
and for any existential formula $\chi(\bar{x}, a)$, it holds that $\vert (\theta \wedge \psi)(\A, a) \vert = \vert(\theta \wedge \psi \wedge \chi)(\A,a)\vert$. 
Then, any realization of $\theta(\bar{x}, a) \wedge \psi(\bar{x}, a)$
satisfies exactly the same existential formulae as $\bar{c}$ over $a$,
and therefore has the same existential type over $a$.

Since $tp_\exists^\A(a/\emptyset)=tp_\exists^\B(b/\emptyset)$, 
we have $\vert (\theta \wedge \psi)(\A, a) \vert = \vert (\theta \wedge \psi)(\B, b) \vert$.
If $\bar{d}\in \B$ is such that $\B \models (\theta \wedge \psi)(\bar{d},b)$,
then $tp_\exists^\A(a,\bar{c}/\emptyset)=tp_\exists^\B(b, \bar{d}/\emptyset)$.
Hence, there is a map $g: a \cup \phi_1(\A, a) \to b \cup \phi_1(\B, b)$ such that $g(a)=b$ and $g$ preserves existential types.

Thus, if $c \in \A$ is such that $\A  \models \bigwedge_{i=1}^m \phi_i(c,a)$,
then the inductive assumption gives  
$\B \models \phi_i(g(c),b)$ for $i=1, \ldots, m$,
so there are at least $n$ many elements $d \in \B$ such that $\B \models \bigwedge_{i=1}^m \phi_i(d,b)$.
If there were more than $n$ many of them, applying the same argument in the other direction,
we would get that $\A \models \bigwedge_{i=1}^m \phi_i(c,a)$ for more than $n$ many elements $c$,
a contradiction.
Thus, $\B \models \Phi(b)$.
\end{proof}

\begin{corollary}\label{multiremark}  
Let $(\K, \preccurlyeq)$ be a multiuniversal AEC, 
and suppose there is a collection of quantifier free formulae $\E$
such that for all $\A \in \K$ and $A \subseteq \A$,
$cl_\A(A)=\E$-$cl_\A(A)$. 
Then, existential types determine Galois types.
Moreover, if $\K$ has AP and JEP and $\A \in \K$ is existentially closed,
then Galois types coincide with existential types in $\A$.
\end{corollary}

\begin{proof}
By Lemmas \ref{etyypit} and \ref{ee},
existential types determine Galois types.
Suppose 
$\K$ has AP and JEP and $\A \in \K$ is existentially closed.
Then, $\K$ has a monster model $\M$ such that $\A \preccurlyeq \M$.
If $\bar{a}, \bar{b} \in \A$ are two (possibly infinite) sequences such that $tp^g(\bar{a}/\emptyset; \A)=tp^g(\bar{b}/\emptyset; \A$),
then there is some automorphism of $\M$ that sends $\bar{a}$ to $\bar{b}$ (cf. Definition \ref{modhom} and the paragraph after). 
Since $\A$ is existentially closed, 
$$tp_\exists^\A(\bar{a}/\emptyset)=tp_\exists^\M(\bar{a}/\emptyset)=tp_\exists^\M(\bar{b}/\emptyset)=tp_\exists^\A(\bar{b}/\emptyset).$$
\end{proof}
  
Before presenting the main result, we recall the definition of a homogeneous class.

\begin{definition}
Let $\k$ be a cardinal.
We say a model $\A$ is strongly \emph{$\k$-homogeneous} if the following holds for all sequences $(a_i)_{i<\k}$,
$(b_i)_{i<\k} \subseteq \A$: If $tp^g((a_i)_{i \in X}; \A)=tp^g((b_i)_{i \in X}; \A)$ for every finite $X \subseteq \k$,
then  there is an automorphism $F$ of $\A$ with $F(a_i)=b_i$
for $i<\kappa$.

We say an AEC $\K$ is \emph{homogeneous}
if it has a strongly $\k$-homogeneous monster model for every cardinal $\k$.
 \end{definition}
 
\begin{corollary}\label{egalois}
Suppose $\K$ is a multiuniversal AEC with AP and JEP,
and there is a collection $\E$ of quantifier free formulae such that for all $\A \in \K$ and $A \subseteq \A$,
$cl_\A(A)=\E$-$cl_\A(A)$.
Then, $\K$ is homogeneous.   
\end{corollary}

\begin{proof}
For each cardinal $\k$, the class $\K$ has a $\k$-universal and $\k$- model homogeneous monster model
$\M_\k$.
By Corollary \ref{multiremark}, in $\M_\k$,
Galois types are the same as existential types.
Since existential types of infinite sequences are determined by the existential types of their finite subsets,
$\M_\k$ is $\k$-homogeneous.
\end{proof}

\section{The AEC framework}

\noindent
We now turn our attention to fields with commuting automorphisms.
We take our signature to be 
$$L=\{0,1, \cdot, +, -,(\,)^{-1}, \sigma_1, \ldots, \sigma_n, \sigma_1^{-1}, \ldots, \sigma_n^{-1}\},$$
and let $T$ be the first order theory that states $K$ is a field,
$\sigma_1, \ldots, \sigma_n$ are automorphisms of $K$, that they commute,
and that for each $i=1, \ldots, n$, the map $\sigma_i^{-1}$ is the inverse of $\sigma_i$.
We denote the field theoretic algebraic closure of a field $K$ by $K^{alg}$.
Moreover, if $K$ is a field and $A \subseteq K$, we use the shorthand $A^{alg}$ for $\langle A \rangle^{alg}$, where $\langle A \rangle$ is the subfield generated by $A$.

One of the main differences between our setting and that of difference fields is that in our case,
existentially closed models need not be algebraically closed as fields.
When we have a field with several automorphisms, each one of the automorphisms extends to the (field-theoretic) algebraic closure but there are many cases in which the lifts cannot be chosen so that they would still commute.
The following example illustrates one such case.

\begin{example}\label{quaternion}
By \cite{quaternion}, there exists a number field $L$ such that $Gal(L/\mathbb{Q})=Q_8$, the quaternion group given by the generating relations 
$$Q_8=\{e, \bar{e}, i, \bar{i}, j, \bar{j}, k, \bar{k}\}=\langle \bar{e},i,j,k \, | \, i^2=j^2=k^2=ijk=\bar{e}, \bar{e}^2=e\rangle.$$
The center of $Q_8$ is $C=\{e, \bar{e}\}$,
and by the fundamental theorem of Galois theory, there is some intermediate field $K$ 
of $L/\mathbb{Q}$ such that $Gal(L/K)=C$.
Then, $Gal(K/\mathbb{Q})=Q_8/C$,
a commutative group whose elements correspond with the cosets of $e$, $i$, $j$, and $k$.
Take the elements corresponding to (say) the cosets of $i$ and $j$.
Possible lifts to $Gal(L/\mathbb{Q})$ are $i$ or $\bar{i}$ and $j$ or $\bar{j}$, respectively,
but none of these commute with each other.
So, these commuting automorphisms of $K$ do not have commuting lifts even to $Gal(L/\mathbb{Q})$,
and thus not to $K^{alg}=\mathbb{Q}^{alg}$ either.
\end{example}

To address this problem, 
we will look at algebraically maximal (see below) models rather than algebraically closed ones.
Eventually, we will work inside a monster model, 
and then the notion of being an algebraically maximal model will coincide with being relatively algebraically closed as a subfield of the monster.

\begin{definition}\label{relalg}
Let $\A \models T$.
We say $\A$ is an \emph{algebraically maximal model of $T$} if the following holds:
Suppose $\B \models T$ and $\A \subseteq \B$,
and let $p(x)$ be a non-zero polynomial in one variable with coefficients from $\A$.
If $b \in \B$ is such that $p(b)=0$, then $b \in \A$.
\end{definition}
 
\begin{remark}\label{perfect}
We note that every algebraically maximal model of $T$ is perfect as a field.
Indeed, let $K \models T$ and suppose $char(K)=p>0$.
The Frobenius map $Frob: x \mapsto x^p$ defines an isomorphism $K^{alg} \to K^{alg}$,
and the perfect closure of $K$ equals $\bigcup_{n \in \mathbb{N}} Frob^{-n}(K)$.
Each distinguished automorphisms $\sigma_i$, $i \in \{1, \ldots, n\}$,
has the unique extension $Frob^{-n} \sigma_i Frob^{n}$ to $Frob^{-n}(K)$.
Clearly, the automorphisms $Frob^{-n} \sigma_i Frob^{n}$, $i=1, \ldots, n$,
commute.
\end{remark}

We use the notion \emph{relatively algebraically closed} in its usual algebraic sense, i.e. if $K \subseteq L$ are fields, we say $K$ is relatively algebraically closed in $L$ if $K^{alg} \cap L=K$.

We want to work inside a monster model, and thus we need to build our AEC in such a way that it has the joint embedding property (JEP) and amalgamation property (AP) (see Definitions \ref{ap} and \ref{jep}). 

In order to obtain the JEP, we will fix a prime model that will be contained in each model.
We will then take the class to consist of all algebraically maximal models of $T$ that contain the prime model.  

In \cite{ChHr}, a closure operator is defined on a difference field by first closing a set under the distinguished automorphism and then taking the (field theoretic) algebraic closure.
We take the same approach, but in order to obtain a model of $T$,
we use the relative algebraic closure instead of the algebraic closure.
 
\begin{definition}
 If $\A$ is an algebraically maximal model of $T$
and $A \subseteq \A$,
then we define $acl_{\sigma}^\A (A)$
to be the smallest subfield of $\A$
that is closed under the distinguished automorphisms and relatively algebraically closed in $\A$. 
\end{definition}

\begin{definition}\label{class}
Suppose $\A$
is a algebraically maximal model of $T$, and
let $\A_0=acl_{\sigma}^\A(\emptyset)$.
We define the class $\K_{\A_0}$
to consist of all algebraically maximal models of $T$ that contain $\A_0$. 
If $\A, \B \in \K_{\A_0}$, we define $\A \preccurlyeq \B$ if $\A \subseteq \B$.
We say a class obtained this way is an \emph{FCA-class} (for Fields with Commuting Automorphisms). 
\end{definition}

Note that at this point we do not yet know whether $\A_0 \in \K_{\A_0}$ since we do not know if $\A_0$
is an algebraically maximal model of $T$.
However, as soon as we prove that we can amalgamate over it (Lemma \ref{AP0}),
it will turn out that it is.

\begin{remark}
Note that if $\A, \B \in \K_{\A_0}$
and $\A \preccurlyeq \B$, then $\A^{alg} \cap \B=\A$.
\end{remark}
 
The idea of the proof for the amalgamation result comes from the proof of Theorem (1.3) in \cite{ChHr},
and it is based on the fact that if two fields, $K_1$ and $K_2$, are linearly disjoint over $K_0$ (inside some large field), 
then the tensor product $K_1 \otimes_{K_0} K_2$ is a domain
(for more on linear disjointness of fields, see e.g. \cite{cohn}, chapter 11.6).

\begin{lemma}\label{tensor}
Let $K_0, K_1, K_2 \models T$ and $K_0 \subseteq K_1 \cap K_2$.
Suppose $\mathbb{F}$ is some field such that $K_1, K_2 \subseteq \mathbb{F}$, and assume
$K_1$ and $K_2$ are linearly disjoint over $K_0$.
Let $\sigma_1, \ldots, \sigma_n$ and $\sigma_1', \ldots, \sigma_n'$
be the distinguished automorphisms of $K_1$ and $K_2$,  respectively.
Suppose $\sigma_i \raj K_0=\sigma_i' \raj K_0$ for $i=1, \ldots, n$,
and suppose these restrictions give the distinguished automorphisms of 
$K_0$.
Then, the automorphisms $\sigma_i$ and $\sigma_i'$, 
have a unique common extension to the field composite of $K_1$ and $K_2$ in $\mathbb{F}$
for  $i=1, \ldots, n$, and these extensions commute. 
 \end{lemma}

\begin{proof}
Let $L$ be the composite of the fields $K_1$ and $K_2$ in $\mathbb{F}$.
Since $K_1$ and $K_2$ are linearly disjoint over $K_0$,
$L$ is the field of quotients of $K_1 \otimes_{K_0} K_2$.
Define automorphisms $\tau_i$, $i=1, \ldots, n$ of $L$ so that
$\tau_i(a \otimes b)=\sigma_i(a) \otimes \sigma_i'(b)$  
for $a \in K_1$, $b \in K_2$.
Now, $\tau_i \raj K_1 = \sigma_i$ and $\tau_i \raj K_2 = \sigma_i'$ for $i=1,\ldots, n$, and the automorphisms clearly commute.
\end{proof}

To prove the amalgamation property, we need the following fact from algebra.

\begin{fact}\label{algebra}
Suppose $E/k$ is a separable field extension, 
and $k$ is relatively algebraically closed in $E$ (in other words, the extension $E/k$ is \emph{regular}).
Then, $E$ and $k^{alg}$ are linearly disjoint over $k$ in $E^{alg}$.  
\end{fact}
\begin{proof}
See e.g. \cite{cohn}, Theorem 11.6.15.
\end{proof}

\begin{lemma}\label{AP0}
Suppose $\K$ is a FCA-class, $K_1 \in \K$, $K_2 \models T$,
and $K_0 \subseteq K_1 \cap K_2$ is such that $K_0 \models T$
and $K_0$ is relatively algebraically closed (as a field) in $K_1$.
Then, there exists some $L \models T$
such that $K_1 \subseteq L$ and an embedding $f: K_2 \to L$ 
such that $f \raj K_0=id$
and $K_1 \cap f(K_2)=K_0.$ 
\end{lemma}

\begin{proof} 
If $\sigma_1, \ldots, \sigma_n$ and $\sigma_1', \ldots, \sigma_n'$
are the distinguished automorphisms of $K_1$ and $K_2$, respectively,
then the distinguished automorphisms of $K_0$ are given by $\sigma_i \raj K_0=\sigma_i' \raj K_0$,
for $i=1, \ldots, n$. 

Embed the fields $K_1$ and $K_2$ (as pure fields) into some large algebraically closed field $\mathbb{F}$.
Since $K_0$ is relatively algebraically closed in $K_1$,
we may assume $K_1$ and $K_2$ are algebraically independent (and thus linearly disjoint) over $K_0^{alg}$
(if needed, use an automorphism of $\mathbb{F}$ to move $K_2^{alg}$ 
while fixing $K_0^{alg})$.
By Remark \ref{perfect}, $K_1$ is a perfect field, 
and since $K_0$ is relatively algebraically closed in $K_1$,
it is also perfect.
Thus $K_1/K_0$ is a separable extension,
and by Fact \ref{algebra},
$K_0^{alg}$ and $K_1$ are linearly disjoint over $K_0$.
Hence, $K_1$ and $K_2$ are linearly disjoint over $K_0$.
By Lemma \ref{tensor}, the automorphisms $\sigma_i$ and  $\sigma_i'$ 
have commuting extensions $\tau_i$ to the composite $L$ of $K_1$ and $K_2$,
and these extensions commute.
Now, $(L, \tau_1, \ldots, \tau_n) \models T$.
\end{proof}

\begin{remark}\label{sulkeumaremark}
Note that it follows from Lemma \ref{AP0}
that if $\K$ is an FCA-class, $\A \in \K$ and $A \subseteq \A$,
then $\A'=acl_\sigma^\A(A) \in \K$.  
Indeed, suppose $\B \models T$, $\A' \subseteq \B$,
$p$ is a non-zero polynomial with coefficients from $\A'$, and $b \in \B$ is such that $p(b)=0$.
By Lemma \ref{AP0},
there is some $\C \models T$ such that $\A \subseteq \C$
and an embedding $f: \B \to \C$ such that $f \raj \A=id$.
Since $\A' \subseteq \A$, $f(b) \in \C$ is a root of $p$,
and since $\A$ is an algebraically maximal model of $T$,
we have $f(b) \in \A$, and thus $f(b)=b$.
Now $b \in \A'$ since $\A'$ is  relatively algebraically closed (as a field) in $\A$.

Thus, in particular, for the field $\A_0$ from Definition \ref{class},
we have $\A_0 \in \K_{\A_0}$.
\end{remark}

It is easy to see that an FCA-class is an abstract elementary class (AEC),
and from Remark \ref{sulkeumaremark},
it follows that it is multiuniversal in the sense of \cite{multi} (see Definition \ref{multiuni} in the present paper).
Indeed, if $\K$ is an FCA-class, $\A, \B \in \K$ are such that $\B \preccurlyeq \A$ and $A \subseteq \B$,
then $acl_\sigma^\A(A) \preccurlyeq \B$,
and thus $cl_\A(A)=acl_\sigma^\A(A)$.
 
FCA-classes serve as  examples of a multiuniversal class that is not universal (as an AEC) (see Example 2.7 in \cite{multi}). 
Moreover, it is easy to see that if we take the collection $\E$ 
to consist of formulae of the form $x^n+\sum_{k=0}^{n-1}t_k(\bar{y})x^k=0$,
where $n>0$, and each $t_k$ is an $L$-term, 
then the $acl_\sigma$-closure coincides with the $\E$-closure from Definition \ref{eclos}.
We will use Corollary \ref{egalois} to show that it is a homogeneous class and that existential types imply Galois types, but first we need to show that it has AP and JEP.

\begin{lemma}\label{AP}
If $\K$ is an FCA-class, then $\K$ has the disjoint amalgamation property.  
\end{lemma}
 
\begin{proof}  
Since $T$ is an $\forall \exists$-theory,
every model of $T$ embeds into an existentially closed model of $T$.
Since existentially closed models of $T$ are algebraically maximal models of $T$, the result follows from
Lemma \ref{AP0}. 
\end{proof}

\begin{remark}\label{freeam}
Note that we have an even stronger property than the disjoint amalgamation property: 
If $\A, \B, \C \in \K$ and there are strong embeddings $f: \C \to \A$ and $f': \C \to \B$,
then there is some $\D \in \K$ and strong embeddings $g: \A \to \D$ and $g': \B \to \D$
such that $g \circ f= g' \circ f'$ and $g(\A)$ and $g'(\B)$ are linearly disjoint (as fields) 
over $g \circ f(\C)$.
\end{remark}

\begin{lemma}\label{JEP}
If $\K$ is an FCA-class, then $\K$ has JEP.
\end{lemma}

\begin{proof}
Since every model in $\K$ contains the model $\A_0$,  
this is a straightforward consequence of Lemma \ref{AP}.
\end{proof}

We will be working in an FCA-class which we, from now on, denote by $\K$.
Since it has AP and JEP, 
there is, for each cardinal $\k$, a $\k$-universal and $\k$- model homogeneous monster model.
We will assume we are working in such a monster model for $\k$ large enough for the monster to contain all models we are interested in.
All elements of  $\K$ (even if small) are not contained in $\M$ but $\M$ will contain isomorphic copies of them.
As is standard practice in model theory, we will assume that the models we consider are strong submodels of $\M$.
This is analogous to the fact that for the purpose of studying countable fields of characteristic $0$, we can consider them as subfields (and, moreover, elementary submodels) of the field of complex numbers even though there, strictly speaking, are proper class many such fields and all of them are not technically subsets of the complex field.  
 
We note that $\M$ is existentially closed in $\K$ (i.e. if there is some larger model $\M' \in \K$
such that $\M \subseteq \M'$, then $\M$ is existentially closed in $\M'$.
Now we can define Galois types as orbits of automorphisms of the monster model. 
If $A \subseteq \M$,
we will denote the set of automorphisms of $\M$ 
that fix $A$ pointwise with $Aut(\M/A)$.
Occasionally, we will be working in the pure field language in a monster model for algebraically closed fields,
which we will denote by $\mathbb{F}$.
We can take it to be the algebraic closure (as a field) of $\M$.

Now, we can use results from the previous section.
We refer the reader to Definitions \ref{eclos} and \ref{eclos2} for $\E^+$-types.
 
\begin{lemma}\label{etyypitmaar}
Let $\E$ be the collection of formulae of the form
$x^n+\sum_{k=0}^{n-1}t_k(\bar{y})x^k=0$,
where $n>0$, and each $t_k$ is an $L$-term.
Then, $\K$ is homogeneous,
and Galois types coincide with $\E^+$-types in every model $\A \in \K$.
Moreover, if $\A \in \K$ is an existentially closed model then, Galois types realized in $\A$
are the same as existential types in $\A$. 
\end{lemma}

\begin{proof}
It is easy to see that for any $\A \in \K$ and $A \subseteq \A$,
$\E$-$cl_\A(A)=acl_\sigma^\A(A)$.
By Lemmas \ref{AP} and \ref{JEP}, $\K$ has AP and JEP, 
so it is homogeneous by Corollary \ref{egalois}.
The other statements follow from Lemma \ref{etyypit}
and Corollary \ref{multiremark}. 
\end{proof}

In particular, Lemma \ref{etyypitmaar} implies that first order types determine Galois types.
For the rest of the paper, we will use existential types as our main notion of type.
We will denote by $S(A)$ the set of existential types over $A$
in the sense of Definition \ref{etyypmaar}
(note that these types $p$ will then be complete in the sense that if $\phi(x)$
is an existential formula with parameters from $A$,
then either $\phi(x) \in p$ or $\neg \phi(x) \in p$).
If $p=tp_\exists^\M(a/A)$, we say $a$ \emph{realises} $p$ or is a \emph{realisation} of $p$.
We say a type is \emph{consistent} if it has a realisation in $\M$.
For the rest of this paper, we will write just $tp_\exists(a/A)$ for $tp_\exists^\M(a/A)$
and $tp^g(a/A)$ for $tp^g(a/A;\M)$.

In AEC frameworks, \emph{bounded closure} is often used as a counterpart for model theoretic algebraic closure.
We say a set $A \subseteq \M$ is \emph{bounded} if $\vert A \vert < \vert \M \vert$,
and a singleton $a \in \M$ is in the \emph{bounded closure} of a set $A \subseteq \M$
if $tp^g(a/A)$ has only boundedly many realisations.
In our setting, boundedness will actually be equivalent with finiteness,
as we shall soon see.

In \cite{ChHr}, it is shown that in models of ACFA, the field theoretic and model theoretic notions of algebraic closure over a substructure coincide (Proposition (1.7)).
The same line of reasoning works also in our setting.
Our analogue for their result is formulated in the following lemma.
It also implies that in our setting, a type has boundedly many realisations if and only if it has finitely many realisations.

\begin{lemma}\label{modalg}
Let $A \subseteq \M$ and let $a=(a_1, \ldots, a_n) \in \M$ be a finite tuple.
\begin{enumerate}[(i)]
\item If $A \models T$,
then $tp^g(a/A)$ has finitely (boundedly) many realisations if and only if $a_i \in A^{alg}$ for $i=1, \ldots, n$.
\item The type $tp^g(a/A)$ has boundedly many realisations if and only if $a \in acl_\sigma(A)$.
\end{enumerate}
\end{lemma}
 
\begin{proof}
In (i), the implication from right to left is clear.
We prove the other direction first in the case $a$ is a singleton.
Suppose $a$ is transcendental over $A$.
Let $b \in \M$ be another transcendental element.
Then $A(a)$ and $A(b)$ are isomorphic as fields,
and thus $A(b)$ can be extended to a model of $T$ 
that is isomorphic with $acl_\sigma(A(a))$.
By Lemma \ref{AP}, they both can be embedded linearly disjointly in some model of $T$. 
This process can be repeated unboundedly many times.
For the general case, we note that if $tp^g(a_1, \ldots, a_n/A)$
has finitely (boundedly) many realisations, 
then so does $tp^g(a_i/A)$ for $1=1, \ldots, n$,
and hence, since the claim holds for singletons,
we have $a_i \in A^{alg}$ for each $i$.

In (ii), the direction from right to left is again clear,
and it suffices to prove the other direction in case $a$ is a singleton. 
Denote  $\A=acl_\sigma(A)$,
and assume $a \notin \A$.
By (i), $tp^g(a/\A)$ (and thus $tp^g(a/A)$)
has unboundedly many realisations.
\end{proof}

\section{Independence}
 
\noindent
Exactly like in 
\cite{ChHr}, we define an independence notion that is based on independence in pure fields and inherits most properties of non-forking from there.
At the end of this section, we will present our version of the independence theorem.
However, we will first show that our independence notion has all the properties of non-forking that we would expect in a simple unstable setting. 
In the next section, they will be used, together with the independence theorem,
to show that a monster model of $\K$ is simple in the sense of \cite{bl}.

\begin{definition}\label{independence}
Let $A, B, C \subseteq \M$.
We say $A$ is \emph{independent} from $B$ over $C$,
denoted $A \da_C B$, if $acl_\sigma(A,C)$ is algebraically independent (as a field)
from $acl_\sigma(B, C)$ over $acl_\sigma(C)$.
\end{definition}

Many of the usual properties of an independence notion follow directly from the corresponding properties for fields.
To see that local character holds, we recall some notions from difference algebra.
 
\begin{definition}
Let $K \models T$, and let $G$ denote the free abelian group generated by the distinguished automorphisms.
Let $R=K\{x_1, \ldots, x_n\}$
be the ring of difference polynomials in  
$n$ variables (i.e. polynomials in the variables $\tau(x_i)$, where $\tau \in G$, $1 \le i \le n$).
We say an ideal of $R$ is a \emph{difference ideal} if it is closed under the distinguished automorphisms.
We say a difference ideal $I$ is \emph{perfect}, if for all $p \in R$,
$\tau_1, \ldots, \tau_r \in G$, and $k_1, \ldots, k_r \in \mathbb{N}$, it holds that if
$\tau_1(p)^{k_1} \cdots \tau_r(p)^{k_r} \in I$, then $p \in I$.
\end{definition}

\begin{lemma}\label{finchar}
Let $a \in \M$ be a finite tuple and $A \subseteq \M$.
Then, there is some finite $A_0 \subseteq A$ 
such that $a \da_{A_0} A$.
\end{lemma}
 
\begin{proof}
Suppose $a$ is an $n$-tuple, i.e. $a \in \M^n$.
Denote $L_A=acl_\sigma(A)$.
Let $I$ be the difference ideal of $L_A\{x_1, \ldots, x_n\}$ generated by the set
$\{p \in L_A\{x_1, \ldots, x_n\} \, | \, p(a)=0\}$.
Now, $I$ is a perfect difference ideal, and by Proposition 2.5.4 of \cite{levin},
there is some finite set $F \subseteq L_A\{x_1, \ldots, x_n\}$
such that $I$ is the smallest perfect difference ideal containing $F$.
Let $B \subseteq L_A$ ($=acl_\sigma(A)$) be the finite set of coefficients of the elements of $F$.
Then, $a \da_B A$, and thus $a \da_B L_A$.

Now, there is some finite $A_0 \subseteq A$ such that $B \subseteq acl_\sigma(A_0)$.
We have $a \da_{A_0} acl_\sigma(A_0)$ and thus $a \da_{A_0} B \cup A_0$.
Since $B \cup A_0 \subseteq L_A$, transitivity for field independence gives us
$a \da_{A_0} L_A$ and thus $a \da_{A_0} A$.
\end{proof}

We now see that our independence notion has all the properties of non-forking that we would expect in a simple unstable setting.

\begin{lemma}\label{nonforkprop}
Let $\M$ be a monster model for $\K$, and suppose $A \subseteq B \subseteq C \subseteq D \subseteq \M$.
Then, the following hold: 
\begin{enumerate}[(i)]
\item (Local character)
For each finite tuple  $a \in \M$, there is some finite $A_0 \subseteq A$ such that $a \da_{A_0} A$. 
\item (Finite character)
If $\bar{a} \in \M$ is a (possibly infinite) tuple and $\bar{a} \nda_A B$, then there is some finite tuple $b \in B$ such that $\bar{a} \nda_A b$.
\item (Extension)
For every (possibly infinite) $\bar{a} \in \M$,
there is some $\bar{b} \in \M$ such that $tp_\exists(\bar{b}/A)=tp_\exists (\bar{a}/A)$
and $\bar{b} \da_A B$.  
\item (Monotonicity)
If $\bar{a} \in \M$ is a (possibly infinite) tuple, and $\bar{a} \da_A D$, then $\bar{a} \da_B C$.
\item (Transitivity)
If $\bar{a} \in \M$ is a (possibly infinite) tuple, $\bar{a} \da_A B$ and $\bar{a} \da_B C$, then $\bar{a} \da_A C$.
\item (Symmetry)
If $\bar{a}, \bar{b} \in \M$ are (possibly infinite) tuples and $\bar{a} \da_A \bar{b}$, then $\bar{b} \da_A \bar{a}.$
\item (Invariance)
If $f$ is an automorphism of $\M$, $A', B', C' \subseteq \M$ and $A' \da_{B'} C'$,
then $f(A') \da_{f(B')} f(C')$.
\end{enumerate}
\end{lemma}

\begin{proof}
(i) is Lemma \ref{finchar}.

For (iii),
we may (by the definition of our independence relation) without loss of generality assume that $acl_\sigma(A)=A$ and $acl_\sigma(B)=B$.
Let $A'$ be an isomorphic (over $A$) copy of $acl_\sigma(A,\bar{a})$
that is algebraically independent (in the field sense) from $B$ over $A$.
Since $A=acl_\sigma(A)$, the field $A$ is relatively algebraically closed in $A'$.
By Remark \ref{perfect}, it is also perfect, so the field extension $A'/A$ is separable and thus regular.
By Lemma 9.9 in \cite{friedjarden}, for regular field extensions, algebraic independence implies linear disjointness. 
Thus, $A'$ is linearly disjoint as a field from $B$ over $A$.
Let $\bar{a}'$ be the image of $\bar{a}$ in $A'$.
By Lemma \ref{tensor}, the distinguished automorphisms on $B$ and $A'$
have common extensions to $A' \otimes_A B$.
Now, there is an embedding $f: A' \otimes_A B \to \M$
such that $f(1 \otimes_A B)=B$, 
and $\bar{b}=f(\bar{a}' \otimes 1)$ is as wanted. 
 
The other properties follow straightforwardly from the fact that they hold for algebraic independence in fields.
\end{proof}

Next, we will prove a version of the independence theorem that allows us to amalgamate three types given that they satisfy certain conditions.
For difference fields, there is a Generalized Independence Theorem (\cite{ChHr}, p. 3009-3010) 
which makes it possible to simultaneously realise any finite number of types over a given base model as long as the tuples realising them are independent over that model. 
When proving the theorem, Chatzidakis and Hrushovski work largely in the setting of pure algebraically closed fields.
At one crucial point, they use the definability of types in $\o$-stable theories to move a tuple of parameters into the base model.
In our setting, models are not algebraically closed as fields,
and thus we do not even know if their theory is stable when we reduce to the pure field language,
so the proof from \cite{ChHr} does not generalise.
Instead, we need some extra assumptions, mainly that one of the types to be combined is nice in the sense of a technical definition that we will present next.
 
In the rest of the paper, we will be working at times in the field $\mathbb{F}=\M^{alg}$,
viewed as a monster model in the \emph{pure field language}.
Thus, when we talk about automorphisms of $\mathbb{F}$,
we mean \emph{field automorphisms} (without the extra structure of distinguished automorphisms),
and they will not necessarily restrict to automorphisms of $\M$.
 
In the following, we will be using the notion of a \emph{strong type}
in the context of $\mathbb{F}$.
By this, we mean the usual first order notion (see \cite{baldwin2}, p. 113, Definition 3.1).
We denote the strong type of $a$ over $A$ by $stp(a/A)$,
and whenever we use strong types, 
we do it in the context of the theory of algebraically closed fields. 
We recall that in this context, $stp(a/A)=stp(b/A)$ if and only if $tp(a/A^{alg})=tp(b/A^{alg})$,
where $tp$ stands for the first order type in the field language.
By a \emph{strong automorphism} of $\mathbb{F}$
over $A$ we mean an automorphism of $\mathbb{F}$
that fixes $A$ pointwise and preserves strong types over $A$
(in our case of algebraically closed fields this is the same as saying that the automorphism fixes
$A^{alg}$ pointwise).
We write $Saut(\mathbb{F}/A)$ 
for the set of all strong automorphisms over $A$. 

\begin{definition}
Let $\A \in \K$, let $x$ and $y$ be possibly infinite
sequences of variables, let $p(x,y) \in S(\A)$,
and suppose that if $(a,b)$ is a realisation of $p$,
then $a \da_\A b$.
Denote by $\mathbb{F}$ 
the algebraic closure (as a field) of $\M$,
viewed as a monster model in the pure field language.  
We say the type $p$ is \emph{nice} if for all finite sets $X \subseteq acl_\sigma(\A, a)$ and 
$Y \subseteq acl_\sigma(\A, a,b)$,
there is some $f \in Saut(\mathbb{F}/X)$ such that
$f(Y) \subseteq acl_\sigma(\A, a)$,
and $f(Y \cap acl_\sigma(\A, b)) \subseteq \A$.
\end{definition}

In the proof of our version of the independence theorem, the niceness assumption comes into play in the form of the following technical lemma. 
 
\begin{lemma}\label{nice}
Let $\A \in \K$, let $x$ and $y$ be possibly infinite sequences of variables, let $p(x,y) \in S(\A)$
be a nice type, and let $(a,b)$ be a realization of $p$.
Let $\B_1=acl_\sigma(\A, a)$, $\B_2=acl_\sigma(\A, b)$,
let $\B_{12}=acl_\sigma(\A, a,b)$,
let $E$ be the field composite of $\B_1$ and $\A^{alg}$,
and let $F$ be the field composite of $\B_2^{alg}$ and $\B_{12}$. 
Then, $E$ is relatively algebraically closed in $F$ (as a field).
\end{lemma}

\begin{proof}
We will work in the pure field language, in $\mathbb{F}$.
Suppose, for the sake of contradiction,
that there is some $a \in (E^{alg} \cap F) \setminus E$.
Let $P(x) \in E[x]$ be the minimal polynomial of $a$ over $E$.
We will derive a contradiction by showing that $P$ has a root in $E$. 
By the definitions of $E$ and $F$, there are finite sets $X_1 \subseteq \B_1$ and $X_2 \subseteq (X_1 \cap\A)^{alg}$
such that $\langle X_1, X_2 \rangle$ 
(the field generated by $X_1$ and $X_2$)
contains the coefficients of $P$,
and finite sets $Y_1 \subseteq \B_{12}$ and $Y_2 \subseteq (Y_1 \cap \B_2)^{alg}$
such that $a \in \langle Y_1, Y_2 \rangle$.
Since the type $p$ is nice,
there is a strong automorphism $f \in Saut(\mathbb{F}/X_1)$
such that $f(Y_1) \subseteq \B_1$
and $f(Y_1 \cap \B_2) \subseteq \A$.
Since $X_2 \subseteq X_1^{alg}$ and $f \in Saut(\mathbb{F}/X_1)$,
the automorphism $f$ is the identity on $\langle X_1, X_2 \rangle$ and therefore $f$ fixes the coefficients of $P$,
so $P(f(a))=0$.
We claim that $f(a) \in E$.
Indeed, $f(a) \in f(\langle Y_1, Y_2 \rangle)$, and
$f(Y_2) \subseteq f((Y_1 \cap \B_2)^{alg})  \subseteq \A^{alg}$,
and since $f(Y_1) \subseteq \B_1$, we have
$f(Y_1 \cup Y_2) \subseteq \langle \B_1, \A^{alg} \rangle=E$.
\end{proof}

Now we can  prove the independence theorem.
When proving it, we will work in the algebraic closure $\mathbb{F}$ of $\M$.
The types given in the statement will give us interpretations of the distinguished automorphisms on certain models in $\K$.
To prove the theorem, we will need to find a model where these automorphisms have common extensions.
For this, we first extend the automorphisms in to the algebraic closures of the original models in such a way that they agree on the overlaps,
and then apply the argument from the proof of the Generalized Independence Theorem in \cite{ChHr}.
From now on, we will denote by $tp_f(a/A)$ the first order type of $a$ over $A$ in the pure field language.

\begin{theorem}\label{ind}
Let $\K$ be an FCA-class, let
$\A \in \K$, let $x_1, x_2, x_3$ be (possibly infinite)
tuples of variables,
and let $p_{12}(x_1,x_2)$, $p_{13}(x_1, x_3)$ and $p_{23}(x_2, x_3)$ be complete existential types over $\A$
such that $p_{12} \raj x_1=p_{13} \raj x_1$, $p_{12} \raj{x_2}=p_{23} \raj x_2$, and $p_{13} \raj x_3=p_{23} \raj x_3$.
Suppose that $p_{12}$ is nice and if $(a_i, a_j)$ realises $p_{ij}$, then  
$a_i \da_\A a_j$.
Then, the type $p_{12} \cup p_{13} \cup p_{23}$ 
can be realised by some tuple $(a_1, a_2, a_3)$
such that $a_1$, $a_2$, and $a_3$ are independent over $\A$.
\end{theorem}
 
\begin{proof}
We will construct a model where the types $p_{12}$, $p_{13}$, and $p_{23}$ 
are realised simultaneously. 
Let $(a_1, a_2)$, $(a_1,a_3')$, and $(a_2,a_3)$ be such that they realise $p_{12}$, $p_{13}$,
and $p_{23}$, respectively,
and $a_3 \da_{\A a_2} a_1$ 
(and hence $a_3 \da_\A \{a_1, a_2\}$ by transitivity).

Denote $\A_i=acl_\sigma(\A, a_i)$ for $i=1,2,3$,
$\A_3'=acl_\sigma(\A, a_3')$,
$\B_{12}=acl_\sigma(\A, a_1, a_2)$,
$\B_{23}=acl_\sigma(\A, a_2, a_3)$,
and $\B_{13}'=acl_\sigma(\A, a_1, a_3')$.  
For each distinguished automorphism $\tau \in \{\sigma_1, \ldots, \sigma_n\}$,
the type $p_{ij}$ gives an interpretation $\tau_{ij}$ 
of $\tau$ on $\B_{ij}$. 
Denote by $\tau_{13}'$ the interpretation of $\tau$ on $\B_{13}'$.
Moreover, there is some $g \in Aut(\M/\A)$ such that $g(\A_3')=\A_3$,
and thus $\tau_{23} \raj \A_3=g \circ (\tau_{13}' \raj \A_3') \circ g^{-1}$.

From now on, we  will work in the pure field language in the algebraic closure $\mathbb{F}$
of $\M$.
First, we extend the map $g$ to an automorphism $f \in Aut(\mathbb{F}/\A^{alg})$ as follows.
Since $\A \in \K$, the field $\A$ is perfect (by Remark \ref{perfect}) and relatively algebraically closed in $\M$,
and thus the fields $\A^{alg}$ and $\M$ are linearly disjoint over $\A$ by Fact \ref{algebra}.
Hence, the automorphisms $g$ of $\M$ and $id$ of $\A^{alg}$
have a common extension $g'$ on the field composite $\A^{alg} \M$.
Extend now $g'$ to an automorphism $f$ of $(\A^{alg} \M)^{alg}=\mathbb{F}$.
We then have $f((\A_3')^{alg})=\A_3^{alg}$.

Thus, for types in the field language, it holds that $tp_f(\A_3^{alg}/\A^{alg})=tp_f((\A_3')^{alg}/\A^{alg})$.
Since $\A_3^{alg}$ is independent (in the field sense) from $\A_1^{alg}$ over $\A^{alg}$ (by the choice of $a_3$),
and the same holds for $(\A_3')^{alg}$ (since $(a_1, a_3') \models p_{13}$), we have (by stationarity in fields)
$tp_f(\A_3^{alg}/\A_1^{alg})=tp_f((\A_3')^{alg}/\A_1^{alg})$.
Hence, there is a field automorphism $F \in Aut(\mathbb{F}/\A_1^{alg})$
such that $F \raj (\A_3')^{alg}=f$.  
Denote now $\B_{13}=F(\B_{13}')$
and let $\tau_{13}=F \circ \tau_{13}' \circ F^{-1}$.
 
We will find extensions 
$\tilde{\tau}_{ij}$ for the maps $\tau_{ij}$ to the models $\B_{ij}^{alg}$
such that they agree on the overlaps.
After that, we will construct a model where the $\tilde{\tau}_{ij}$ are compatible,
just like it is done in the proof of the Generalized Independence Theorem in \cite{ChHr}, p. 3009-3010.
To give the big picture, we now sketch the idea of constructing the extensions and then go into the details in the next paragraph.
We will first extend the $\tau_{ij} \raj \A$ to $\A^{alg}$.
Then, since $\A^{alg}$ and $\A_3$ are linearly disjoint over $\A$
(note that this is a consequence of the fact that $\A, \A_3 \in \K$
since this guarantees that $\A$ is perfect by Remark \ref{perfect} 
and relatively algebraically closed in $\A_3$ 
which allows us to use Fact \ref{algebra}),
we can find a common extension of this automorphism and $\tau_{13} \raj \A_3$ to $\A^{alg} \A_3$ which then extends to $\A_3^{alg}$.
Since $\A_3^{alg}$ and $\B_{13}$ are linearly disjoint over $\A_3$
(again because $\A_3, \B_{13} \in \K$),
this automorphism and $\tau_{13}$ have a common extension to $\A_3^{alg} \B_{13}$
and thus to $\B_{13}^{alg}$.
With an analogous process, we will extend $\tau_{23}$ to an automorphism of $\B_{23}^{alg}$.
Finally, we will apply Lemma \ref{nice} to construct $\tilde{\tau}_{12}$. 

We now go into the details of the above argument.
Let $\tau_{\A^{alg}}$ be an extension of the 
 $\tau_{ij} \raj \A$ to $\A^{alg}$.
Now, $\tau_{\A^{alg}}$ and $\tau_{13} \raj \A_3$ have a common extension to
$\A^{alg} \A_3$ which then extends to an automorphism $\tau_{\A_3^{alg}}$ of $\A_3^{alg}$.
The automorphisms $\tau_{\A_3^{alg}}$ and $\tau_{13}$ have a common extension to $\A_3^{alg} \B_{13}$
which extends to an automorphism $\tilde{\tau}_{13}$ of $\B_{13}^{alg}$.
We have $\tau_{23} \raj \A_3=\tau_{\A_3^{alg}} \raj \A_3$,
and these automorphisms have  a common extension to $\A_3^{alg} \B_{23}$
since $\A_3^{alg}$ and $\B_{23}$ are linearly disjoint over $\A_3$ (again because $\A_3, \B_{23} \in \K$).
It extends to an automorphism $\tilde{\tau}_{23}$ of $\B_{23}^{alg}$.
 
By Lemma \ref{nice}, we know that $\A_1 \A^{alg}$ is relatively algebraically closed in $\A_2^{alg} \B_{12}$.
It is also the composite of two perfect fields ($\A_1$ is perfect since $\A_1 \in \K$) and thus perfect.
Hence, $\A_1^{alg}$ and $\A_2^{alg} \B_{12}$ are linearly disjoint over $\A_1 \A^{alg}$.
We will now construct an automorphism $\sigma$ on $\A_2^{alg} \B_{12}$ such that $\tau_{12} \subseteq \sigma$
and $\sigma \raj \A_1 \A^{alg}=\tilde{\tau}_{13} \raj \A_1 \A^{alg}$,
and then simultaneously extend $\sigma$ and $\tilde{\tau}_{13} \raj \A_1^{alg}$
to an automorphism of $\A_1^{alg} \A_2^{alg} \B_{12}$ which then extends to an automorphism $\tilde{\tau}_{12}$ 
on $\B_{12}^{alg}$. 
For this, we note that $\A_2$ is perfect and relatively algebraically closed in $\B_{12}$
(since $\A_2, \B_{12} \in \K$),
and thus $\tau_{12}$ and $\tilde{\tau}_{23} \raj \A_2^{alg}$ have a common extension $\sigma$ to
$\A_2^{alg} \B_{12}$. 
Since $\A_1^{alg}$ and $\A_2^{alg} \B_{12}$ are linearly disjoint over $\A_1 \A^{alg}$, the maps $\tilde{\tau}_{13} \raj \A_1^{alg}$ and $\sigma$ extend to a common automorphism of $\A_1^{alg}\A_2^{alg}\B_{12}$, which in turn extends to an automorphism $\tilde{\tau}_{12}$ of $\B_{12}^{alg}$.
  
We now proceed just like in the proof of the Generalised Independence Theorem in \cite{ChHr},
p. 3009-3010.
First, we note that by the choice of $a_3$,
we have $\B_{12} \da_{\A_1} \B_{13}$,
and thus $\B_{12}^{alg}$ and $\B_{13}^{alg}$ are independent in the field sense over $\A_1^{alg}$.
By Lemma 9.9. in \cite{friedjarden}, they are linearly disjoint over $\A_1^{alg}$,
and thus the automorphisms 
$\tilde{\tau}_{12}$ and $\tilde{\tau}_{13}$ (which agree on $\A_1^{alg}$) have a common extension $\tau'$ 
to the field composite $F_1:=\B_{12}^{alg}\B_{13}^{alg}$.
Let $F_0$ be the field composite of $\A_2^{alg}$ and $\A_3^{alg}$
in $F_1$ (note that these are linearly disjoint over $\A^{alg}$).
By the choice of $a_3$, the tuple $a_3$ is independent in the field sense from $\B_{12}^{alg}$
over $\A^{alg}$, and thus, by Remark 1.9(2) in \cite{ChHr},
$F_1 \cap \B_{23}^{alg}=F_0$,
which implies
 that $F_1$ and $\B_{23}^{alg}$ are linearly disjoint over $F_0$.
Thus the automorphisms $\tau'$ and $\tilde{\tau}_{23}$
have a common extension to the field composite $L:=\B_{12}^{alg}\B_{13}^{alg}\B_{23}^{alg}$
of $F_1$ and $\B_{23}^{alg}$ in $\mathbb{F}$.

Restrict the automorphisms obtained this way to the field $L'$ generated by 
$\B_{12}$, $\B_{13}$, and $\B_{23}$ within $L$.
Since the restrictions commute on these fields, they also commute on $L'$.   
Extend $L'$ to an algebraically maximal model of the theory of fields with commuting automorphisms.
This proves the theorem. 
\end{proof} 

One motivation behind the present paper is to eventually use the methods of geometric stability theory in FCA-classes.
In a geometric context, there often is a specific type that is the object of study.
For example, one could look at the ``line" given by a type of rank 1.
We end this section with a lemma which guarantees that by extending the base model,
we can always assume that a given type is nice in the sense of Definition \ref{nice}.
This means that after taking a free extension of the type, 
we can always apply Theorem \ref{ind} to combine it with other types that satisfy the relevant assumptions. 

\begin{lemma}\label{laajenee}
Suppose $\K$ is an FCA-class and $\M$ is a monster model for $\K$.
Let $a \in \M$ and $\A \in \K$.
There is some $\A' \in \K$ of cardinality $\vert \A \vert + \vert a \vert$ such that 
\begin{itemize}
\item $\A \subseteq \A'$;
\item $a \da_\A \A'$;
\item if $b \in \M$ is such that $b \da_{\A'} a$,
then $tp_\exists(a,b/\A')$ is nice.
\end{itemize}
\end{lemma}

\begin{proof}
Consider quadruples $(X, Y, Z, b)$
where $X$ is a finite subset of $acl_\sigma(\A,a)$, 
$Z \subseteq Y$ are finite subsets of $\M$,
and $b$ is a finite tuple from $\M$. 
We say two such quadruples $(X_i, Y_i, Z_i, b_i)$ and $(X_j, Y_j, Z_j, b_j)$ 
are isomorphic if $X_i=X_j$, and in the field $\mathbb{F}=\M^{alg}$ the strong type (in the field language)
$stp(b_i,Y_i, Z_i/X_i)=stp(b_j,Y_j, Z_j/X_i)$.  

Let now $(X_i, Y_i,Z_i, b_i)_{i<\lambda}$
list, up to isomorphism, all such quadruples.
We note that $\lambda \le \vert a \vert + \vert \A \vert$ 
(there are at most $\aleph_0$ many strong types over a finite set).
Construct models $\A_i$, $i<\lambda$, as follows.
Let $\A_0=\A$.
For each $i<\lambda$, construct $\A_{i+1}$
as follows.
If there is some $(X^*, Y^*, Z^*,b^*)$
which is isomorphic to $(X_i, Y_i, Z_i, b_i)$, and with
$b^* \da_{\A_i} a$,
$Y^* \subseteq acl_\sigma(\A_i, a,b^*)$, 
and $Z^* \subseteq acl_\sigma(\A_i, b^*)$,
let $\A_{i+1}=acl_\sigma(\A_i, b_i^*)$.
Otherwise, let $\A_{i+1}=\A_i$.
At limit steps, take unions.

Let $\A^1=\bigcup_{i<\lambda} \A_i$.
Repeat the above process with $\A^1$ in place of $\A$,
continue this way to obtain a chain of models $\A^i$, $i<\o$,
and set $\A'=\bigcup_{i<\o} \A^i$.
We claim that $\A'$ is as wanted.

First of all, by symmetry, transitivity and local character of the independence relation
(Lemma \ref{nonforkprop}, (ii), (v), and (vi)),
it follows from the construction that $a \da_\A \A'$.

Let now $b \in \M$ be a (possibly infinite) tuple such that $b \da_{\A'} a$.
We show $tp_\exists(a,b/\A')$ is nice.
Let $X \subseteq acl_\sigma(\A', a)$
and $Y \subseteq acl_\sigma(\A',a,b)$ be finite sets,
and let $Z=Y \cap acl_\sigma(\A',b)$.
Now, there is some finite $b_0 \subseteq b$ such that $Y \subseteq acl_\sigma(\A',a,b_0)$
and $Z \subseteq acl_\sigma(\A', b_0)$.
By Lemma \ref{nonforkprop} (i) and (v), there is a finite set $A \subseteq \A'$
such that $b_0 \da_A \A' a$.
There is some
$n<\o$ such that $A \subseteq \A^n$ (and thus $a \da_{\A^n} b_0$), $X \subseteq acl_\sigma(\A^n, a)$,
$Y \subseteq acl_\sigma(\A^n, a,b_0)$ and $Z \subseteq acl_\sigma(\A^n, b_0)$.
Let $(X_i^n, Y_i^n, Z_i^n,b_i^n)_{i<\lambda}$ 
be the list we have used to construct $\A^{n+1}$ from $\A^n$,
and let $(\A_i^n)_{i<\lambda}$ 
be the construction of $\A^{n+1}$.
Now, there is some $i<\lambda$ 
such that $(X,Y,Z,b_0)$ is isomorphic to $(X_i^n, Y_i^n, Z_i^n,b_i^n)$,
and then $(X,Y,Z,b_0)$ witnesses $\A_i^n \neq \A_{i+1}^n$.
So there is $(X^*, Y^*, Z^*, b^*)$
such that $\A_{i+1}^n=acl_\sigma(\A_i^n,b_i^*)$, $Y^* \subseteq acl_\sigma(\A_i^n,a,b^*)$,
$Z^* \subseteq acl_\sigma(\A_i, b^*)$,
and $(X^*, Y^*, Z^*, b^*)$ is isomorphic to $(X_i^n, Y_i^n, Z_i^n,b_i^n)$
and thus also to $(X,Y,Z,b_0)$. 
Hence, there is an automorphism $f \in Saut(\mathbb{F}/X)$
such that $f(Y)=Y^*$ and $f(Z)=Z^*$.
Since $b^* \in \A'$, the automorphism $f$ is as wanted,
and $tp_\exists(a,b/\A')$ is nice.
\end{proof}

\section{Simplicity}\label{last}

\noindent 
In \cite{bl}, the first order notion of a simple theory is generalised to the non-elementary framework of homogeneous models.
We will show in this section that a monster model for an FCA-class is simple 
(or more specifically, $\aleph_0$-simple or supersimple) in the sense of \cite{bl}.
The idea of our proof comes from
\cite{KP}, where it is shown that a first order theory is simple if and only if it has a syntactic (i.e. invariant under automorphisms) notion of independence with the usual properties of non-forking 
(the same that are listed in our Lemma \ref{nonforkprop})  
and satisfies the Independence Theorem over models.
We will adapt the argument to our setting and show that in an FCA-class,
Theorem \ref{ind} together with Lemma \ref{nonforkprop} implies simplicity.
The main difference will be that we cannot use Compactness the way it is used in the first order context.
 
We will now recall the definition of simplicity from \cite{bl}.
Since we aim to show that our class is $\aleph_0$-simple (also called supersimple), 
we will only provide the definitions relevant to that. 
We refer the reader to \cite{bl}, Definitions 2.1-2.5 for a more general notion of simplicity in the context of homogeneous structures.

We note that in \cite{bl}, everything happens, strictly speaking, 
inside a fixed model rather than a class of models.
Moreover, \cite{bl} assumes that in their model, quantifier free types imply Galois types.
We may also make this assumption after expanding our language with suitable relation symbols,
as is done in Remark \ref{multire}. 

\begin{definition}
Let $\kappa$ be an ordinal and $A \subseteq \M$.
We say a sequence $(a_i)_{i<\k}$ is \emph{indiscernible over $A$}
if for all $n<\o$ and $i_1< \cdots < i_n<\k$ 
and $j_1< \cdots < j_n<\k$, it holds that $tp_\exists(a_{i_1}, \ldots, a_{i_n}/A)=tp_\exists(a_{j_1}, \ldots, a_{j_n}/A)$.
\end{definition} 
 
\begin{definition}\label{dividing}
We say an existential type $p(v, b)$
\emph{divides} over $A \subseteq \M$,
if there is an infinite $A$-indiscernible sequence $\{b_i \, | \, i \in \kappa\}$
for some infinite ordinal $\k$,
with $tp_\exists(b_0/A)=tp_\exists(b/A)$,
such that $\bigcup_{i \in \k} p(v, b_i)$ is inconsistent.
\end{definition}

\begin{definition}\label{free}
If $A, B, C \subseteq \M$,  
we say $A$ is \emph{$\aleph_0$-free} 
from $B$ over $C$ if for all tuples $a \in A$  and $b \in B \cup C$ such that $\vert a \vert, \vert b \vert < \kappa$,
$tp_\exists(a/b)$ does not divide over $C$.
\end{definition}
 
Next, we present the notion of  $\aleph_0$-simplicity (or supersimplicity) in the sense of \cite{bl}
(adapted to our setting).
In addition to the two conditions of the following definition (local character and free extensions property),
Definition 2.5 in \cite{bl} also requires a third condition for $\kappa$-simplicity, namely that $\kappa$ -freeness has finite character in the sense of Definition 2.3 in \cite{bl} (note that finite character in the sense of this definition is slightly different from finite character in the sense of e.g. our Lemma \ref{nonforkprop}).
However, $\aleph_0$-freeness always has finite character in this sense, so the condition is void in the   context of $\aleph_0$-simplicity, and we thus omit it.  
  
\begin{definition}\label{blsimple}
Let $\M$ be a monster model for a homogeneous AEC.
We say $\M$ is \emph{$\aleph_0$-simple (supersimple)} if the following hold:
\begin{itemize}
\item If $a \in \M$ is a finite tuple and $A \subseteq \M$,
then there is some finite $A_0 \subseteq A$ such that $a$ is $\aleph_0$-free from $A$ over $A_0$;
\item If $a \in \M$, $A \subseteq \M$,
$tp_\exists(a/A)$ has infinitely (unboundedly) many realisations, $B \subseteq A$  
is such that $a$ is $\aleph_0$-free from $A$ over $B$,
and $C \subseteq \M$ is such that $B \subseteq C$,
then there is some $c \in \M$ realising $tp_\exists(a/A)$
such that $c$ is $\aleph_0$-free from $C$ over $B$.
\end{itemize}

\end{definition}

\begin{remark}
Another notion of simplicity in a non-elementary setting can be found in \cite{pi} (Definition 3.2).
There, dividing is defined just as in \cite{bl} (see our Definition \ref{dividing})
but with $\omega$ in place of $\kappa$
(thus it is a stronger requirement than in \cite{bl}).
Then, a type $p(x)$ is defined to \emph{fork} over a set $A$
if there is a (possibly infinite) set $\Phi(x)$ of existential formulae (with parameters) each of which divides over $A$ such that $\M \models p(x) \rightarrow \bigvee \Phi(x)$.
Finally, $\M$
is said to be \emph{simple} if for any finite tuple $a \in \M$ and any set $A \subseteq \M$,
there is some countable $A_0 \subseteq A$ such that $tp_\exists(a/A)$ does not fork over $A_0$.

We note that if $\M$ is a monster model for an FCA-class and $\aleph_0$-simple in the sense of \cite{bl}, then it is simple in the sense of \cite{pi}.
Indeed, let $a \in \M$ be a finite tuple, and $A \subseteq \M$. 
We need to find a countable set $A_0 \subseteq A$ such that $p=tp(a/A)$
does not fork over $A_0$.

Suppose first that $p$ has infinitely many realizations.
By the first property in Definition \ref{blsimple} (local character),
there is some finite $A_0 \subseteq A$ such that $a$ is $\aleph_0$-free from $A$ over $A_0$.
We show that $p$ does not fork over $A_0$.
Assume towards a contradiction that it does. 
Then, there is a collection of existential formulae $\Phi$ such that $p \vdash \bigvee \Phi$ and each $\phi(x,c) \in \Phi$ divides (in the sense of \cite{pi} and thus also in the sense of Definition \ref{dividing}) over $A_0$. 
By the extension property in Definition \ref{blsimple}, there are some $a' \models p$ and $B \supseteq A$ such that $a'$ is free from $B$ over $A_0$, and for all $\phi(x,b) \in \Phi$, $b \in B$.
Denote $q=tp(a'/B)$.
Now there is some $\phi(x,b) \in \Phi \cap q$.
It follows that $a'$ is not $\aleph_0$-free from $Ab$ over $A_0$ (and thus not from $B$ either), a contradiction.
 
Suppose now $p$ has only finitely many realizations.
By Lemma \ref{modalg}, $a \in acl_\sigma(A)$.
Now, there is some finite $A_0 \subseteq A$ such that $a \in acl_\sigma(A_0)$.
We claim that $p$ does not fork over $A_0$.
If it does, then there is some $\phi(x,b)$
such that $\M \models \phi(a,b)$
and $\phi(x,b)$ divides over $A_0$ (in the sense of \cite{pi}).
Let $(b_i)_{i<\omega}$ be an indiscernible sequence that witnesses the dividing.
For each $i<\o$,
there is some $a_i \in \M$
such that $tp(a_ib_i/A_0)=tp(ab/A_0)$.
Since $p=tp(a/A_0)$ has only finitely many realizations, there is some infinite set $X \subseteq \omega$
such that $a_i=a_j=a'$ for all $i,j \in X$ and some $a' \in \M$.
By homogeneity of $\M$,
there is some automorphism $f \in Aut(\M/A_0)$ which takes the sequence $(b_i)_{i \in \o}$ to $(b_i)_{i \in X}$.
Now, $f^{-1}(a')$ realizes $\bigcup_{i<\o}\phi(x,b_i)$, a contradiction.

We have chosen to use the notions of dividing and simplicity in \cite{bl} rather than the ones in \cite{pi}
since the notion of dividing is more general and the notion of simplicity is stronger. 
Moreover, we feel they fit our setting better. 
\end{remark} 

Next, we will show that $\aleph_0$-freeness coincides with the independence relation defined in the previous section.
Simplicity will then follow from Lemma \ref{nonforkprop}. 
There is a related argument about the connection between dividing and independence in \cite{KP}
(Claims I and II in the proof of Theorem 4.2), and we will modify it to our setting.
In \cite{KP}, Claim II states that independence implies non-dividing,
and its proof uses Compactness twice.

The proof begins with stretching an infinite indiscernible sequence
to the length $\kappa$ for some large $\kappa$,
and then using Ramsey's Theorem and Compactness to find an increasing, 
continuous sequence of models such that each model contains the beginning of the sequence and the rest of the sequence is indiscernible over the said model.
Here, we will circumvent Compactness by using the Erd\"os-Rado Theorem.
In the first order setting, Compactness is used for a second time at the end of the proof,
to deduce that a type is consistent by showing that each finite subtype is.
There, in place of using Compactness, we will we apply Theorem \ref{ind} and move things around in the monster.
These arguments are captured in the following three lemmas.
Another difference to \cite{KP} is that when we use the Independence Theorem,
we need to take care that the niceness condition holds.
This can be arranged when we are dealing with an indiscernible sequence (see Lemma \ref{yhdistys1}).
  
\begin{lemma}\label{l2}
Let $A \subseteq \M$, and let $(b_i)_{i\le\kappa}$ be an infinite, non-constant sequence of finite tuples that is indiscernible over $A$.
Then, there is an increasing, 
continuous sequence of models $\A_i$, $i<\kappa$,
such that for each $i$,
\begin{enumerate}[(i)]
\item $\A_i$ contains $A \cup (b_j:j<i)$;
\item $(b_j : i \le j \le \kappa)$ is indiscernible over $\A_i$.
\end{enumerate}
\end{lemma}

\begin{proof}
Since we are working in a homogeneous structure, indiscernible sequences can be extended (see e.g. Lemma 1.5, (ii) in \cite{bl}, and we can extend the sequence  $(b_i)_{i<\kappa}$  
to length $\beth_{(2^{\kappa+\v A \v})^+}$. 
Add now to our language Skolem functions that give roots for those polynomials that have a root in $\M$
 (i.e. for $A \subseteq \M$, we will have $SH(A)=acl_\sigma(A)$, where $SH$ denotes the Skolem hull).
The usual Ehrenfeucht-Mostowski construction (for details, see e.g.Theorem 8.18 and Appendix A in \cite{baldwin}) gives us an indiscernible sequence
$(b_i')_{i<\kappa}$ and a model $\A=SH((b_i')_{i<\kappa}, A)$  such that if $i_1< \ldots <i_n<\k$, then there are some $j_1< \ldots <j_n<\k$ such that
$SH(b_{i_1}', \ldots, b_{i_n}') \cong SH(b_{j_1}, \ldots, b_{j_n})$.
If $L$ is the original signature (without the Skolem functions),
then $\A \raj L \in \K$.
For $i_1<\ldots<i_n<\kappa$, 
there are some $j_1<\ldots<j_n<\kappa$ such that
\begin{eqnarray*}
SH(b_{i_1}', \ldots, b_{i_n}', A) \raj L \cong SH(b_{j_1}, \ldots, b_{j_n},A) \raj L \cong SH(b_{i_1}, \ldots, b_{i_n},A) \raj L,
\end{eqnarray*}
and thus, by homogeneity (Lemma \ref{etyypitmaar}), $\A \raj L \cong acl_\sigma(A,b_i, i<\k)$.
Now, the models $\A_i=acl_\sigma(A, (b_j)_{j<i})$, $i<\k$ are as wanted.

\end{proof}

\begin{lemma}\label{yhdistys1}
Let $\kappa$ be an infinite cardinal, let $\A \in \K$,
and let the sequence $I=(b_i:i<\kappa)$ 
of possibly infinite tuples be indiscernible and independent over $\A$.
Let $p(x, y) \in S(\A)$ be an existential type such that for some $a \in \M$, it holds that $a \da_\A b_0$ and $(a,b_0) \models p$.
 Then, there is a type $q(x,y,z) \in S(\A)$
such that $q \raj (x,y)= q \raj (x,z)=p$, $q \raj (y,z)=tp_\exists(b_0,b_1/\A)$,
and if $(a, c_1, c_2) \models q$, then $a \da_\A c_1c_2$.
\end{lemma}

\begin{proof}
Since indiscernible sequences can be extended (see e.g. Lemma 1.5, (ii) in \cite{bl}),
we may without loss of generality assume that $\kappa\ge \omega+1$.
Denote $\A^*=acl_{\sigma}(\A, b_i)_{i<\o}$,
and let $p'(x,y)$ be a free extension of $p$ to $\A^*$.
Now, the type $tp_\exists(b_{\omega+1}, b_{\o}/\A^*)$
is nice;
indeed, if we are given a finite set $S \subseteq \o$,
then there is a permutation of the sequence $(b_i)_{i<\kappa}$
that will fix $\{b_i \, | \, i \in S\} \cup \{b_{\omega+1}\}$
and take $b_{\o} \mapsto b_n$
for some $n<\o$ such that $n>i$ for all $i \in S$,
and this permutation extends to a strong automorphism of the algebraic closure $\mathbb{F}$
of $\M$.
To see that it extends, we note that the sequence $(b_i)_{i<\kappa}$
is indiscernible over $\A$ also in $\mathbb{F}$,
so its members have the same strong type over $\A$ and hence over $\A^{alg}$.
Applying Theorem \ref{ind} (with $p_{12}(x_1,x_2)=tp_\exists(b_{\o+1}, b_\o/\A^*)$ and $p_{13}(x_1,x_3)=p_{23}(x_2,x_3)=p'(y,x)$, i.e. the type which we obtain from $p'(x,y)$ by permuting the two variables)
we  obtain, after a permutation of variables, a type $q'(x,y,z) \in S(\A^*)$ such that $q' \raj (x,y)=p'(x,y)$, $q' \raj(x,z)=p'(x,z)$, and $q' \raj (y,z)=tp_\exists(b_\o,b_{\o+1}/\A^*)$.
Now $q=q' \raj \A$ is as wanted.
\end{proof}

\begin{lemma}\label{yhdistys2}
Let  $\A \in \K$,
and let $(b_i:i<\k)$
be a sequence of possibly infinite tuples that is indiscernible and independent over $\A$.
Let $p(x, y) \in S(\A)$ be an existential type such for some $a \in \M$, it holds that $a \da_\A b_0$ and $(a,b_0) \models p$.
Then, the type $q=\bigcup_{i<\k} p(x, b_i)$ is consistent.
\end{lemma}

\begin{proof}
Consider first the sequence $(b_i)_{i<\o}$.
Applying Lemma \ref{yhdistys1} to the type $p$, we obtain a type $p_1'(x_1,x_2,x_3) \in S(\A)$ such that $p_1'\raj(x_1,x_2)=p(x_1,x_2)$, $p_1' \raj (x_1,x_3)=p(x_1,x_3)$, and $p_1' \raj(x_2,x_3)=tp_\exists(b_0,b_1/\A)$.
Then, there is  some $a_1 \in \M$ such that $(a_1, b_0, b_1) \models p_1'$ and $a_1 \da_\A b_0 b_1$.
Next, we apply Lemma \ref{yhdistys1} to the type $p_1'$ and the sequence $((b_{2i},b_{2i+1}))_{i<\o}$ (which is independent and indiscernible over $\A$) to obtain a type $p_2'(x_1, x_2,x_3,x_4, x_5) \in S(\A)$ and $a_2 \in \M$ such that 
 $(a_2, b_0, b_1, b_2, b_3) \models p_1'$.

By applying Lemma \ref{yhdistys1} repeatedly this way, we obtain consistent types $p_i \in S(\A, b_0, \ldots, b_{2^i-1})$
for $1 \le i <\o$,
such that $p_i \subseteq p_{i+1}$, 
$p_{i} \raj (\A, b_{j})=p(x,b_j)$ for $0 \le j \le 2^i-1$,
and if $a_i$ realizes $p_i$, then $a_i \da_\A b_0 \cdots b_{2^i-1}$.

For each $1 \le i < \o$, 
let $a_i$ realise  $p_i(x, b_0, \ldots, b_{2i-1})$.
We have $tp_\exists(a_1/b_0,b_1, \A)=tp_\exists(a_2/b_0,b_1, \A)$,
and thus there is some automorphism $f_1 \in Aut(\M/\A, b_0, b_1)$
such that $f_1(a_2)=a_1$.
Since 
$$tp_\exists(a_3/b_0, b_1, b_2, b_3 \A)=tp_\exists(a_2/b_0, b_1, b_2, b_3, \A),$$
therefore also
$$tp_\exists(f_1(a_3), f_1(b_2),f_1(b_3)/\A b_0b_1)=tp_\exists(a_1, f_1(b_2),f_1(b_3)/\A b_0b_1),$$
and we can find an automorphism $f_2 \in Aut(\M/\A, b_0, b_1, f_1(b_2),f_1(b_3))$
such that $f_2(f_1(a_3))=a_1$.
Continuing this way (the next stage fixes $b_0$, $b_1$, $f_1(b_2)$, $f_1(b_3)$, $f_2 \circ f_1(b_4)$, $f_2 \circ f_1(b_5)$,
$f_2 \circ f_1(b_6)$, and $f_2 \circ f_1(b_7)$, and sends $f_2 \circ f_1(a_4)$ to $a_1$)
we construct a sequence of functions $f_i$, $1 \le i <\o$,
satisfying:
\begin{itemize}
\item $f_i \circ f_{i-1} \circ \cdots  \circ f_1(a_{i+1})=a_1$;
\item $f_i \circ f_{i-1} \circ \cdots  \circ f_1(b_j)=f_{i-1} \circ \cdots  \circ f_1(b_j)$, when $j<2^i$.
\end{itemize}

Define now a map $F_\omega: \A \cup (b_i)_{i<\o} \to \M$ by setting $F_\omega \raj \A =id$, $F_\omega(b_0)=b_0$,
and for each $n$ such that $0<n<\omega$,  
$$F_\omega(b_n)=f_i \circ f_{i-1} \circ \cdots \circ f_1(b_n),$$
where $i$ is the smallest number such that $n<2^{i}$.
Note that then $f_i(F_\omega(b_n))=F_\omega(b_n)$ whenever $n<2^{i-1}$.  
By Corollary \ref{egalois}, $F_\omega$ extends to an isomorphism from $cl_\sigma(\A, b_i)_{i<\o}$
to $cl_\sigma(\A, F_\o(b_i))_{i<\o}$,
and thus to some $F_\o^* \in Aut(\M, \A)$.

Let $c=(F_\o^*)^{-1}(a_1)$.
Now, for each $n<\omega$, we have 
$$tp_\exists(c, b_0, \ldots, b_n/\A)=tp_\exists(a_1, F_\omega(b_0), \ldots, F_\omega(b_n)/\A )=tp_\exists(a_{n+1}, b_0, \ldots, b_n/\A),$$
where the last equality follows from the fact that if $n>1$ and $i\le n$,
then $f_n \circ f_{n-1} \circ \cdots  \circ f_1(a_{n+1})=a_1$
and $f_{n} \circ f_{n-1} \circ \cdots  \circ f_1(b_i)=F_\omega(b_i)$
(note that the equality also holds if $n\in \{0,1\}$, since then $F_\omega(b_n)=b_n$). 
Thus $c$ realises $q=\bigcup_{i<\o} p(x, b_i)$ and $c \da_\A (b_i)_{i<\o}$.

Let $(b_i')_{i<\k}$
be a sequence where $b_0'=(b_i)_{i<\o}$,
$b_1'=(b_i)_{\o \le i <2\o}$, and so on.
By homogeneity, it is indiscernible over $\A$.
Moreover, it is independent over $\A$, and thus we can apply the same process as above.
Continuing this way, we eventually get that $q=\bigcup_{i<\k} p(x, b_i)$ is consistent.
\end{proof}

We are  now ready to prove that our independence notion coincides with freeness (see Definition \ref{free}),
and it will then follow that $\K$ is $\aleph_0$-simple.
The proof is as in \cite{KP} (Claims I and II of Theorem 4.2), and we provide it for the sake of exposition. 

\begin{lemma}\label{l3}
Let $a \in \M$ be a finite tuple, and $A, B \subseteq \M$.
Then, $a \da_A B$ if and only if $a$ is free from $B$ over $A$.
\end{lemma}

\begin{proof}
We prove first the direction from right to left.
If $a \nda_A B$,
then, by Lemma \ref{nonforkprop} (ii), there is some finite tuple $b \in B$
such that $a \nda_A b$.
Now, $b \notin acl_\sigma(A)$,
so $p=tp_\exists(b/A)$ has unboundedly many realisations by Lemma \ref{modalg}.  
Just like in the proof of lemma \ref{l2},
we can apply Erd\"os-Rado to obtain a sequence $(b_i)_{i<\o}$
of realisations of $p$ that are independent and indiscernible over $A$ (see also e.g.Theorem 8.18 and Appendix A in \cite{baldwin})
Next, 
show that if $q(x,b)=tp_\exists(a/bA)$,
then $r=\bigcup_{i<\o}q(x, b_i)$ is inconsistent, and thus $q$ divides over $A$.
If $r$
was consistent and realised by some $a' \in \M$,
then $a' \nda_A b_i$ for all $i<\o$,
which implies $a' \nda_{A \cup \{b_j \, | \, j<i\}} b_i$ for all $i<\o$ 
(otherwise symmetry and transitivity would give $b_i \da_A a'$).
However, by the local character of the independence relation (Lemma \ref{nonforkprop}, (ii)),
there is some finite $A_0 \subseteq A$ and some $i<\omega$ such that
$a' \da_{A_0 \cup \{b_j \, | \, j<i\}} A \cup \{b_i \, | \, i<\o\}$,
and thus by monotonicity, $a' \da_{A \cup \{b_j \, | \, j<i\}} b_i$, a contradiction.
Hence, $a$ is not free from $B$ over $A$.

For the other direction, we may without loss of generality assume $A \subseteq B$.
Suppose $a \da_A B$, and let $b \in B$.
We need to show that $p=tp_\exists(a/Ab)$ does not divide over $A$.
Let $\k$ be an infinite cardinal,
and let $(b_i : i\le \k)$ be an indiscernible sequence over $A$,
with $b_0=b$.
By Lemma \ref{l2},
there is an increasing, continuous sequence of models $\A_i \in \K$ for $i<\k$
such that 
\begin{enumerate}[(i)]
\item $\A_i$ contains $A \cup (b_j:j<i)$;
\item $(b_j: i \le j \le \k)$ is indiscernible over $\A_i$.
\end{enumerate}
Denote $\B=\bigcup_{i<\k} \A_i$.
By local character (Lemma \ref{nonforkprop}, (ii)),
there is some finite set $A_0 \subseteq \B$
such that $b_\k \da_{A_0}  \B$.
Since $A_0 \subseteq \A_i$ for some $i<\k$,
it follows from monotonicity (Lemma \ref{nonforkprop}, (iv) )
that  $b_\k \da_{\A_i} \B$,
and therefore, since the sequence $(b_j: i \le j \le \k)$ 
is indiscernible over $\A_i$,
it is also independent over $\A_i$.

Denote $\A=\A_i$.
After relabeling, 
we have a sequence $(b_i)_{i<\k}$ 
that is independent and indiscernible over $\A$.
The type $p(x, b_0) \in S(Ab_0)$ has a free extension $p'(x,b_0)$ to $\A b_0$ by Lemma \ref{nonforkprop}, (iii). 
Let $a'$ be a realisation of $p'$.
Since $a' \da_A b_0$ and $a' \da_{Ab_0} \A$,
we have by transitivity and monotonicity  (Lemma \ref{nonforkprop}, (v) and (iv))
that $a' \da_\A b_0$.  
Now, write $p'$ as $p''(x,b_0)$ (i.e. $p''(x,y) \in S(\A)$).
By Lemma \ref{yhdistys2}, the type $\bigcup_{i<\k} p''(x,b_i)$ is consistent.
\end{proof}

\begin{corollary}\label{SIMPLE}
Let $\K$ be an FCA-class, and let $\M$ be a monster model for $\K$.
Then $\M$ is $\aleph_0$-simple (in the sense of Definition \ref{blsimple}). 
\end{corollary} 

\begin{proof}
Follows directly from Lemma \ref{l3} and Lemma \ref{nonforkprop}.
\end{proof}

\vspace{0.5cm}

\noindent
\textbf{Open questions}

\begin{enumerate}[1.]
\item Is it possible to get rid of the niceness condition in Theorem \ref{ind} and prove the full independence theorem over models?
\item Is it possible to improve Lemma \ref{laajenee} to find a model $\A'$ which would work for all $a,b \in \M$ which are independent over $\A'$?
\end{enumerate}

\end{document}